\documentclass{amsart}
\usepackage[all]{xy}
\usepackage{amsmath,amsfonts,amssymb,amsthm,epsfig,amscd,comment}

\newtheorem{Theorem}{Theorem}[section]
\newtheorem{Proposition}[Theorem]{Proposition} 
\newtheorem{Lemma}[Theorem]{Lemma}

\newtheorem{Corollary}[Theorem]{Corollary}

\theoremstyle{definition}
\newtheorem{Remark}[Theorem]{Remark}

\def\k{{\Bbbk}}
\def\id{{\mathrm{I}}}
\def\l{{\lambda}}
\def\L{{\sf{L}}}
\def\R{{\sf{R}}}
\def\sl{\mathfrak{sl}}
\newcommand{\Z}{\mathbb{Z}}
\newcommand{\C}{\mathbb{C}}
\newcommand{\p}{\mathbb{P}}
\newcommand{\Cone}{\operatorname{Cone}}
\newcommand{\Span}{\operatorname{Span}}

\newcommand{\D}{\mathcal{D}}
\def\O{{\mathcal O}}
\def\sE{{\mathcal{E}}}
\def\sF{{\mathcal{F}}}
\def\sG{{\mathcal{G}}}
\def\sT{{\mathcal{T}}}
\def\bG{\mathbb{G}}
\newcommand{\E}{\mathsf{E}}
\newcommand{\F}{\mathsf{F}}
\newcommand{\bF}{\mathbb{F}}
\newcommand{\G}{\mathsf{G}}
\renewcommand{\H}{\mathsf{H}}
\newcommand{\U}{\mathsf{U}}
\newcommand{\T}{\mathsf{T}}
\newcommand{\la}{\langle}
\newcommand{\ra}{\rangle}
\newcommand{\Ext}{\operatorname{Ext}}
\newcommand{\Hom}{\operatorname{Hom}}
\newcommand{\End}{\operatorname{End}}
\newcommand{\Inf}{\operatorname{Inf}}
\newcommand{\rank}{\operatorname{rank}}

\begin{document}

\title[Derived equivalences via categorical $sl_2$ actions]{Derived equivalences for cotangent bundles of Grassmannians via categorical $sl_2$ actions}

\author{Sabin Cautis}
\email{scautis@math.harvard.edu}
\address{Department of Mathematics\\ Rice University \\ Houston, TX}

\author{Joel Kamnitzer}
\email{jkamnitz@math.toronto.edu}
\address{Department of Mathematics\\ University of Toronto \\ Toronto, ON Canada}

\author{Anthony Licata}
\email{amlicata@math.stanford.edu}
\address{Department of Mathematics\\ Stanford University \\ Palo Alto, CA}

\begin{abstract}
We construct an equivalence of categories from a strong categorical sl(2) action, following the work of Chuang-Rouquier.  As an application, we give an explicit, natural equivalence between the derived categories of coherent sheaves on cotangent bundles to complementary Grassmannians.
\end{abstract}

\date{\today}
\maketitle
\tableofcontents

\section{Introduction}

\subsection{Categorical $ \sl_2 $ actions}

A {\bf categorical $\sl_2$ action} consists of a sequence of $\k$-linear additive categories $\D(-N), \dots, \D(N)$ together with functors, for $ \l \in \Z $,
$$\E(\l): \D(\l-1) \rightarrow \D(\l+1) \text{ and } \F(\l): \D(\l+1) \rightarrow \D(\l-1)$$
satisfying the relations 
\begin{equation}\label{eq:1}
\E(\l-1) \circ \F(\l-1) \cong \id^{\oplus \l}_{\D(\l)} \oplus \F(\l+1) \circ \E(\l+1) \text{ if } \l \ge 0.
\end{equation}
\begin{equation}\label{eq:2}
\F(\l+1) \circ \E(\l+1) \cong \id^{\oplus -\l}_{\D(\l)} \oplus \E(\l-1) \circ \F(\l-1) \text{ if } \l \le 0
\end{equation}

On the complexified (split) Grothendieck groups  $V(\l) := K(\D(\l)) \otimes_\Z \C$, the functors $\E(\l)$ and $\F(\l)$ induce maps of vector spaces $e(\l) := [E(\l)]$ and $f(\l) := [F(\l)]$.  By the above relations, $ V = \oplus V(\l) $ is a locally finite representation of $ \sl_2(\C) $, where $ e = \left[ \begin{smallmatrix} 0 & 1 \\ 0 & 0 \end{smallmatrix} \right] $ acts by $ \oplus_\l e(\l) $ and $ f = \left[ \begin{smallmatrix} 0 & 0 \\ 1 & 0 \end{smallmatrix} \right] $ acts by $ \oplus_\l f(\l) $.

Any such $\sl_2$ action on $\oplus_\l V(\l)$ integrates to an $SL_2$ action.  Hence the reflection element 
$$t = \left[ \begin{matrix} 0 & -1 \\ 1 & 0 \end{matrix} \right] $$
induces an isomorphism of vector spaces $V(\l) \xrightarrow{\sim} V(-\l)$.  An interesting question is whether this isomorphism can be lifted to an equivalence of categories $ \T: \D(\l) \rightarrow \D(-\l) $ coming from a lift of the reflection element.

\subsection{Construction of the equivalence}
In their seminal paper \cite{CR}, Chuang and Rouquier constructed such an equivalence under the assumption that the underlying weight space categories are abelian, and the categorical $\sl_2$ action is by exact functors.  More precisely, the constructed an equivalence between the corresponding derived categories.

In our motivating geometric examples, however, the natural functors are not exact, and as a result the weight space categories $ \D(\l) $ are triangulated instead of abelian.  
Nevertheless, it is still natural to ask for an equivalence between $ \D(\l) $ and $ \D(-\l) $.

The setup we use in order to modify the Chuang-Rouquier construction is that of \emph{strong categorical $\sl_2$ actions} (section \ref{sec:strongcat}). In particular, this means that categories $\D(\l)$ are graded and the 
$\E$s and $\F$s satisfy a graded version of relations (\ref{eq:1}) and (\ref{eq:2}). The grading implies that on the level of the Grothendieck group there is an action of the quantum group $U_q(\sl_2)$.  

As an additional data, we demand functors 
$$\E^{(r)}(\l): \D(\l-r) \rightarrow \E^{(r)}(\l+r) \text{ and } \F^{(r)}(\l): \D(\l+r) \rightarrow \D(\l-r)$$
which categorify the elements $e^{(r)}= \frac{e^r}{r!}$ and $f^{(r)}=\frac{f^r}{r!}$ along with natural transformations $ X, T$ which are used to rigidify the isomorphisms (\ref{eq:1}) and (\ref{eq:2}).

Even though our axioms are different than those of Chuang-Rouquier, our construction of the complex which induces the functor $\T $ and our proof that it is an equivalence is adapted from \cite{CR}.  Here is a sketch of the construction.
 
If we restrict $t \in SL_2$ to the weight space $V(\l)$ with $\l \ge 0$, then one can write $t$ as 
$$t = f^{(\l)} - f^{(\l+1)}e + f^{(\l+2)}e^{(2)} \pm \dots$$
where the sum is finite since $V(\l)=0$ for $\l \gg 0$. We can lift $f^{(\l+s)} e^{(s)} $ to the composition of functors $\F^{(\l+s)} \circ \E^{(s)}$. We then form a complex of functors 
$\Theta_*$ whose terms are 
$$\Theta_s := \F^{(\l+s)}(s) \circ \E^{(s)}(\l+s)\langle -s\rangle.$$
(the $\langle \cdot\rangle$ represents a shift in the grading). The connecting maps of the complex are defined via various adjunction morhisms (section \ref{sec:complex}).

The main result of this paper, Theorem \ref{th:equiv}, states that if our categories $\D(\l)$ are triangulated then the convolution of this complex is an equivalence $\T: \D(\l) \xrightarrow{\sim} \D(-\l)$.

\subsection{Application to stratified Mukai flops}

Our main application in this paper is to construct equivalences between derived categories of coherent sheaves on cotangent bundles of Grassmannians.  We fix a positive integer $N$ and let $\D(\l) := DCoh(T^\star \bG(k,N))$ be the bounded derived category of coherent sheaves on the cotangent bundle of the Grassmannian where $\l =N-2k$. In \cite{ckl2} we constructed a strong categorical $\sl_2$ action on these categories, where the functors $ \E, \F $ act by natural correspondences.

Hence by the main result of this paper, we get a non-trivial, natural equivalence 
\begin{equation}\label{eq:3}
DCoh(T^\star \bG(k,N)) \xrightarrow{\sim} DCoh(T^\star \bG(N-k,N))
\end{equation}
(see Corollary \ref{cor:mainequiv}). 

When $k=1$, $T^\star \bG(1,N)$ and $T^\star \bG(N-1,N)$ differ by a standard Mukai flop and the equivalence we produce is already known (see \cite{kaw1} or \cite{nam1}). 

When $k > 1$ the varieties $T^\star \bG(k,N)$ and $T^\star \bG(N-k,N)$ are related by a stratified Mukai flop. The problem of constructing such an equivalence (\ref{eq:3}) in these cases was posed by Namikawa in \cite{nam2}.  Our work gives the first solution to this problem with the exception of the case $\bG(2,4)$ where an equivalence was constructed by Kawamata in \cite{kaw2}.   We should note that the idea of using categorical $ \sl_2 $ actions to construct the equivalence (\ref{eq:3}) was first proposed by Rouquier in \cite[section 4.4.2]{RBourbaki}.

Despite the indirect nature of our construction, we are able to give a fairly concrete description of the equivalence, which we do in section \ref{sec:application}.  In particular we prove that it is induced by a Fourier-Mukai kernel which is a Cohen-Macauley sheaf.

Similar geometric settings for equivalences constructed from strong categorical $\sl_2$ actions, including convolutions of affine Schubert varieties and Nakajima quiver varieties, are considered in \cite{ckl1} and \cite{ckl3}, respectively.

\subsection{Acknowledgements}
First and foremost, we would like to thank Joe Chuang and Raphael Rouquier for their generous explanations.  Without this assistance this paper would not have been possible.  We also benefited from conversations with Roman Bezrukavnikov, Christopher Brav, Daniel Huybrechts and Yoshinori Namikawa.

S.C. was partially supported by National Science Foundation Grant 0801939. He also thanks MSRI for its support and hospitality during the spring of 2009.  J.K. was partially supported by a fellowship from the American Institute of Mathematics. A.L. would like to thank the Max Planck Institute in Bonn for support during the 2008-2009 academic year.

\section{Equivalences from strong categorical $\sl_2$ actions}

In this section we begin by reviewing the concept of a strong categorical $\sl_2$ action. Then we describe how to form a complex of functors from such an action. Finally we state our result (Theorem \ref{th:equiv}) which explains when and how this complex of functors gives an equivalence of categories. 

\subsection{Strong categorical $\sl_2$ actions}\label{sec:strongcat}

We begin by reviewing the definition of a strong categorical $\sl_2$ action (of nil affine Hecke type).

Let $\k$ be a field. A {\bf strong categorical $\sl_2$ action} consists of the following data.
\begin{enumerate}
\item A sequence of $\k$-linear $\Z$-graded, additive categories $\D(-N), \dots, \D(N)$. Graded means that each category $\D(\l)$ has a shift functor $\la \cdot \ra$ which is an equivalence.   
\item Functors 
$$\E^{(r)}(\l): \D(\l-r) \rightarrow \D(\l+r) \text{ and } \F^{(r)}(\l): \D(\l+r) \rightarrow \D(\l-r)$$
where $\l \in \Z$. We will usually write $\E(\l)$ for $\E^{(1)}(\l)$ and $\F(\l)$ for $\F^{(1)}(\l)$ while one should think of $\E^{(0)}(\l)$ and $\F^{(0)}(\l)$ as the identity functor $\id$ on $Y(\l)$.
\item Morphisms
$$\eta: \id \rightarrow \F^{(r)}(\l)  \E^{(r)}(\l) \la r \l \ra \text{ and } 
\eta: \id \rightarrow \E^{(r)}(\l)  \F^{(r)}(\l) \la -r\l \ra$$
$$\varepsilon: \F^{(r)}(\l) \E^{(r)}(\l) \rightarrow \id \la r\l \ra \text{ and }
\varepsilon: \E^{(r)}(\l)  \F^{(r)}(\l) \rightarrow \id \la -r\l \ra. $$
\item Morphisms
$$\iota : \F^{(r+1)}(\l) \la r \ra \rightarrow \F(\l+r) \F^{(r)}(\l-1) \text{ and } \pi : \F(\l+r) \F^{(r)}(\l-1) \rightarrow \F^{(r+1)}(\l) \la -r \ra.$$
\item Morphisms
$$X(\l): \F(\l) \la -1 \ra \rightarrow \F(\l) \la 1 \ra \text{ and }
T(\l): \F(\l+1) \F(\l-1) \la 1 \ra \rightarrow \F(\l+1) \F(\l-1) \la -1 \ra.$$
\end{enumerate}
Let $\Hom(\D(\l), \D(\l'))$ be the category of additive functors which commute with the shifts.  This is also an additive category. 


On this data we impose the following additional conditions.

\begin{enumerate}

\item The morphisms $ \eta $ and $ \varepsilon $ are units and counits of adjunctions
\begin{enumerate}
\item $\E^{(r)}(\l)_R = \F^{(r)}(\l) \la r\l \ra$ for $r \ge 0$
\item $\E^{(r)}(\l)_L = \F^{(r)}(\l) \la -r\l \ra$ for $r \ge 0$
\end{enumerate}
\item $\F$'s (and by adjointness the $\E$'s) compose as 
$$\F^{(r_2)}(\l+r_1) \F^{(r_1)}(\l-r_2) \cong \F^{(r_1+r_2)}(\l) \otimes_{\k} H^*(\bG(r_1,r_1+r_2))$$
For example,
$$\F(\l+1) \F(\l-1) \cong \F^{(2)}(\l) \la -1 \ra \oplus \F^{(2)}(\l) \la 1 \ra.$$
In the case $r_1=r$ and $r_2=1$ we also require that the maps
$$\oplus_{i=0}^r (X(\l + r)^i  \id) \circ \iota \la -2i \ra : \F^{(r+1)}(\l) \otimes_\k H^\star(\p^r) \rightarrow \F(\l+r) \F^{(r)}(\l-1)$$
and
$$\oplus_{i=0}^r \pi \la 2i \ra \circ (X(\l+r)^i \id) : \F(\l+r) \F^{(r)}(\l-1) \rightarrow \F^{(r+1)}(\l) \otimes_\k H^\star(\p^r)$$ 
are isomorphisms.  We also have the analogous condition when $r_1=1$ and $r_2=r$. 

\begin{Remark} Intuitively, $\iota$ maps into the ``bottom'' factor of 
$$\F(\l+r)  \F^{(r)}(\l-1) \cong \F^{(r+1)}(\l) \otimes_\k H^\star(\p^r)$$
while $\pi$ maps out of the ``top'' factor. 
\end{Remark}

\item If $ \l \ge 0 $ then
\begin{equation*}
\E(\l-1) \F(\l-1) \cong \F(\l+1) \E(\l+1) \oplus \id \otimes_{\k} H^\star(\p^{\l-1}).
\end{equation*}
The isomorphism is induced by
$$\sigma + \sum_{j=0}^{\l-1} (I X(\l-1)^j) \circ \eta: \F(\l+1)  \E(\l+1) \oplus \id \otimes_{\k} H^\star(\p^{\l-1}) \xrightarrow{\sim} \E(\l-1)  \F(\l-1)$$
where $\sigma$ is the composition of maps
\begin{eqnarray*}
\F(\l+1)  \E(\l+1) & \xrightarrow{\eta  I  I} & \E(\l-1)  \F(\l-1)  \F(\l+1)  \E(\l+1) \la 1-\l \ra \\
& \xrightarrow{I  T({\l})  I} & \E(\l-1)  \F(\l-1)  \F(\l+1)  \E(\l+1) \la -1-\l \ra \\
& \xrightarrow{I I \epsilon} & \E(\l-1)  \F(\l-1).
\end{eqnarray*}
Similarly, if $\l \le 0$ then 
\begin{equation*}
\F(\l+1)  \E(\l+1) \cong \E(\l-1)  \F(\l-1) \oplus \id \otimes_{\k} H^\star(\p^{-\l-1})
\end{equation*}
with the isomorphism induced in the same way as above. 

\item The $X$s and $T$s satisfy the nil affine Hecke relations:
\begin{enumerate}
\item $T(\l)^2 = 0$
\item $(I T(\l-1)) \circ (T(\l+1)  I) \circ (I  T(\l-1)) = (T(\l+1)  I) \circ (I  T(\l-1)) \circ (T(\l+1)  I)$ as endomorphisms of $ \E(\l-2)\E(\l)\E(\l+2)$.
\item $(X(\l+1)  I) \circ T(\l) - T(\l) \circ (I X(\l-1)) = I = - (I  X(\l-1)) \circ T(\l) + T(\l) \circ (X(\l+1)  I)$ as endomorphisms of $ \E(\l-1)\E(\l+1) $.
\end{enumerate}

\item $\Hom(\F^{(r)}(\l), \F^{(r)}(\l) \la i \ra) = 0$ if $i < 0$ while $\End(\F^{(r)}(\l)) = \k \cdot \id$ for $r \ge 0$. 
\end{enumerate}

\begin{Remark}
The above definition is stated in terms of the functors $\F $.  In \cite{ckl2}, the definition was stated in terms of $ \E $, but by adjunction these definitions are equivalent.
\end{Remark}

\subsection{Action of $ U_q(\sl_2)$ on the Grothendieck group}

Let $ V_{\mathrm{sp}}(\l) = K(\D(\l)) $ denote the split complexified Grothendieck group of $ \D(\l) $.  
$V_{\mathrm{sp}}(\l)$ is a $ \C[q,q^{-1}] $ module, where $ q $ acts by the shift operator $ \langle -1 \rangle $. 

The functors $ E^{(r)}(\l) $ induce linear maps $ e^{(r)}(\l) : V(\l-r) \rightarrow V(\l+r) $ and similarly for $ F^{(r)} $.  Immediately from the definition, we see that this defines an action of $ U_q(\sl_2) $ on $ \oplus_\l V(\l) $ where the weight spaces are the $V(\l) $.

Since  $ \oplus_\l V(\l) $ is a locally finite $ U_q(\sl_2) $ module, the quantum Weyl group element $ t \in \widehat{U_q(\sl_2)} $ acts and gives us isomorphisms $ V(\l) \rightarrow V(-\l) $ for each $ \l $.  The action of this reflection element is described by the following result.

\begin{Proposition} \label{th:tact}
For $ \l \ge 0 $, restricted to $ V(\l) $, the element $ t $ acts by
\begin{equation*}
f^{(\l)} - q f^{(\l+1)} e + q^2 f^{(\l+2)} e^{(2)} \pm \dots
\end{equation*}
\end{Proposition}
\begin{proof}
The action of $ t $ on $ V(\l) $ for $ \l \ge 0 $ is characterized by the following properties:
\begin{enumerate}
\item If $ e v = 0$ and $ v \in V(\l) $, then $ t v = f^{(\l)} v $,
\item $t f = - e k^{-1} t = -q^2 k^{-1} e t $
\end{enumerate}
Property (i) is clear. To check property (ii) we compute both sides on a vector $v$ of weight $ \l + 2 $.  Here $ [k] = \frac{q^k - q^{-k}}{q-q^{-1}} $ will denote a quantum integer. First for the left hand side we get
\begin{align*}
\sum_{s=0}^\infty (-1)^s q^s f^{(\l+s)} e^{(s)} f v &= \sum (-1)^s q^s f^{(\l+s)} ([\l + s +1]e^{(s-1)} + f e^{(s)}) v \\
&= \sum (-1)^s q^s [\l+s+1] (f^{(l+s)}e^{(s-1)} + f^{(\l+s+1)}e^{(s)})v \\
&= - \sum (-1)^s q^{2s+\l+2} f^{(l+s+1)} e^{(s)}v
\end{align*}
where in the last step we have used the telescoping nature of the sum and the equation $$ q^s[\l+s+1] - q^{s+1}[\l+s+2] = -q^{2s+\l+2} .$$

The right hand side gives
\begin{align*}
-q^2 k^{-1} \sum_{s=0}^\infty (-1)^s q^s e f^{(\l+2+s)} e^{(s)} v &= -q^2k^{-1} \sum (-1)^s q^s ([s+1] f^{(\l+s+1)} + f^{(\l+s+2)}e) e^{(s)}  v  \\
&= - q^2 k^{-1} \sum (-1)^s q^s [s+1](f^{(\l+s+1)} e^{(s)} + f^{(\l+s+2)} e^{(s+1)}) v \\
&= -q^2 k^{-1} \sum (-1)^s q^{2s} f^{(\l+s+1)} e^{(s)} v \\
&= - \sum (-1)^s q^{2s+\l +2} f^{(\l+s+1)} e^{(s)} v 
\end{align*}

Hence we see that property (ii) holds and this completes the proof.
\end{proof}

\subsection{The complex}\label{sec:complex}

Given a strong categorical $\sl_2$ action we can construct, for each $ \l \ge 0$, a complex of functors $ \Theta_*(\l): \D(\l) \rightarrow \D(-\l) $.  Note that we use homological indexing for our complex in order to simplify the notation.

The terms in the complex are 
$$ \Theta_s(\l) = \F^{(\l+s)}(s)\E^{(s)}(\l+s)\langle -s\rangle $$
where $ s = 0, \dots, (N - \l)/2$. The differential $d_s: \Theta_s \rightarrow \Theta_{s-1} $ is given by the composition of maps
\begin{equation*}
\F^{(\l+s)}\E^{(s)}\langle -s\rangle 
\xrightarrow{\iota \iota} \F^{(\l+s-1)} \F\E \E^{(s-1)}\langle -(\l+s-1)-(s-1)-s\rangle \\
 \xrightarrow{\varepsilon} \F^{(\l+s-1)}\E^{(s-1)}\langle -s+1\rangle.
\end{equation*}

Note that by Proposition $ \ref{th:tact}$, we have$ t|_{V(\l)} = \sum_s (-1)^s [\Theta_s(\l)] $ on the Grothendieck group.

Chuang-Rouquier defined an analogous complex in \cite{CR}, which they call Rickard's complex.

\begin{Remark}
The above complex gives can be considered as a functor between various categories associated to opposite weight spaces.  In particular, we could consider $\Theta_*(\l)$ as a functor between a category of complexes over $ \D(\l) $ to a category of complexes over $ \D(-\l) $.  This is what is done in \cite{CR}, where they prove that $\Theta_*(\l)$ induces an equivalence of homotopy categories of complexes under the assumption that each weight category $ \D(\l) $ is abelian.  More recently, Rouquier \cite{Rnew} proves a similar statement under the milder assumption that each $ \D(\l) $ is 
$ \k$-linear.  The direction we choose is a bit different, however, in that we use the complex $\Theta_*(\l)$ to define a triangulated equivalence between $ \D(\l) $ and $ \D(-\l)$ under the assumption that our weight categories $\D(\l)$ are triangulated.
\end{Remark}

\subsection{The main Theorem}
In order to use $ \Theta_* $ to make a functor between $ \D(\l) $ and $ \D(-\l) $, we need to be able to take cones of functors.  Unfortunately, this is not immediately easy to do, due to the non-functoriality of cones in triangulated categories.  To avoid this problem, we will make the following assumptions:

\begin{enumerate}
\item Each $ \D(\l) $ is a triangulated category, and all functors $ \E^{(p)}, \F^{(p)} $ are exact.  
Note that in this situation, the categories $ \D(\l) $ have two shift functors, one $ \langle 1 \rangle $ coming from the structure of strong categorical $ \sl_2 $ action and another $ [1] $ coming from the structure of triangulated category.  These shifts need not coincide, and we let $ \{1\} = [1] \langle -1 \rangle $. 

\item For each pair $ \l, \l'$, there exists a graded triangulated category $ \D(\l, \l') $ and an additive functor $\Phi: \D(\l, \l') \rightarrow \Hom(\D(\l), \D(\l')) $, denoted $ \sE \mapsto \Phi_\sE $.  We assume that $ \Phi $ commutes with both shifts $ \{1\} $ and $ [1]$.  We further assume that there is a monoidal structure 
$ * : \D(\l, \l') \times \D(\l', \l'') \rightarrow \D(\l, \l'') $ such that $ \Phi $ intertwines this operation with the composition of functors. Moreover, if $ \sE \rightarrow \sF \rightarrow \sG \rightarrow \sE[ 1] $ is a distinguished triangle in $ \D(\l, \l') $, then for any $ A \in \D(\l)$, we require that $ \Phi_\sE(A)  \rightarrow \Phi_\sF(A) \rightarrow \Phi_\sG(A) \rightarrow \Phi_\sE(A)[ 1] $ be a distinguished triangle in $ \D(\l') $. We assume that $\D(\l,\l')$ satisfies the Krull-Schmidt property (any object decomposes uniquely into irreducibles). 

\begin{Remark} This Krull-Schmidt property can be achieved by assuming that $\D(\l, \l')$ is idempotent complete and that (each graded piece of ) the hom space between any two objects is finite dimensional (see section \ref{sec:infmaps}). 
\end{Remark}

\item For each functor $ \E^{(p)}(\l)$, we choose $\sE^{(p)}(\l) \in \D(\l-p, \l+p) $ such that $\Phi_{\sE^{(p)}}(\l) = \E^{(p)}(\l) $.   Similarly, we choose lifts of $ \F^{(p)}(\l) $ and of all morphisms mentioned in the definition of strong categorical $\sl_2 $ action.  These lifts should satisfy the same conditions as in the definition of strong categorical $ \sl_2 $ action.
We further assume that $ \sF^{(r)}(\l)\la r \l \ra * (\cdot) : \D(\mu, \l +r) \rightarrow \D(\mu, \l -r) $ is the right adjoint to the functor $ \sE^{(r)}(\l) * (\cdot) : \D(\mu, \l -r) \rightarrow \D(\mu, \l+r)  $ (and similarly for the other pairs of adjunctions).
\end{enumerate}

When these assumptions holds, we will say that we have an \textbf{enhanced} strong categorical $ \sl_2 $ action.  

\begin{Remark}
Alternatively, we could have assumed that for each $ \D(\l) $, is the homotopy category of some pre-triangulated DG category $ \widetilde{\D(\l)} $,  in which case we would also have to assume that we have chosen lifts of our functors and morphisms to the DG level. 
\end{Remark}

We will not actually use these assumptions until section \ref{se:mainproof}.  When we do use these assumptions, we will do so implicitly.  That means, we will not switch notation to $ \sE $, etc, but just make use of the fact that we can take cones of morphisms of functors.

The category $ \D(\l) $ also has a complexified Grothendieck group as a triangulated category, denoted $ V_{\mathrm{tr}}(\l) $.  It is naturally a $ \C[q, q^{-1}]$-module where $ -q $ acts by $ \{1\} $.
The direct sum of these Grothendieck groups $ \oplus_{\l} V_{\mathrm{tr}}(\l) $ is a $ U_q(\sl_2) $ module and the natural map $ \oplus_\l V_{\mathrm{sp}}(\l) \rightarrow \oplus_\l V_{\mathrm{tr}}(\l) $ is a $ U_q(\sl_2) $ module map.  

\begin{Remark}
In our main geometrical examples we have $ \D(\l) = D^{\C^\times}(Y(\l)) $ and 
$ \D(\l, \l') = D^{\C^\times}(Y(\l) \times Y(\l')) $, the bounded derived categories of $ \C^\times$-equivariant coherent sheaves on varieties $ Y(\l) $ and their products.  In these cases, $ \{1 \} $ will shift the equivariant structure.  
\end{Remark}

When we have an enhanced strong categorical $ \sl_2$ action, we can consider the complex $ \Theta_* $ as a complex in the triangulated category $ \D(\l, -\l) $.  As a complex in a triangulated category, we can take convolutions of this complex (see section \ref{se:conv}).  Then we have the following main theorem, whose proof occupies the rest of the paper.

\begin{Theorem} \label{th:equiv}
The complex $ \Theta_* $ has a unique convolution $ \T $.  $ \T $ gives an equivalence between $ \D(\l) $ and $ \D(-\l)$.  The map between Grothendieck groups $ V_{\mathrm{tr}}(\l) \rightarrow V_{\mathrm{tr}}(-\l) $ induced by $ \T $ equals the reflection element $ t $.
\end{Theorem}

The last statement of Theorem \ref{th:equiv} follows immediately from Proposition \ref{th:tact}.

\subsection{Relation to spherical twists and outline of the proof of the main theorem}
A special case of this situation is related to the notion of spherical twists.  Let us assume that $ \D(\l) $ is zero except for $ \l = -2, 0,2 $ and also assume that $ \D(-2) = \D(2), \E(-1) = \F(1) = \G, \F(-1) = \E(1) = \H $, so that we have the situation
\begin{equation*}
\D(-2) \underset{\H}{\overset{\G}{\rightleftarrows}} \D(0) \underset{\G}{\overset{\H}{\rightleftarrows}} \D(2)
\end{equation*}

Then the conditions of strong categorical $ \sl_2 $ action imply that $ \G_R = \H\langle-1\rangle, \G_L = \H\la 1\ra $ and $$ \H\G\cong \id_{\D(2)}\langle -1\rangle \oplus \id_{\D(2)}\langle 1\rangle.$$  In particular, this means that $ \G $ is a spherical functor, a notion due (in varying levels of generality) to Seidel-Thomas \cite{ST},  Horja \cite{H}, Anno \cite{A}, and Rouquier \cite{R2}.

The above complex $ \Theta_*(0) : \D(0) \rightarrow \D(0) $ is just 
$$\Theta_*(0) =  \G \H \langle -1\rangle \rightarrow \id $$
in this case.  From the theory of spherical twists, we know that $ \T = \Cone( \G\H\langle -1\rangle \rightarrow \id) $ is an equivalence.  So one way to think about the convolution of the complex $ \Theta_* $ is as a generalized spherical twist.

This perspective gives us the inspiration for the proof the main theorem.  Recall that in the proof that spherical twists give equivalences one considers two sets of objects in $\D(0)$: those killed by $\H $ and those in the image of $ \G$ (see for example \cite{H} proof of Proposition 8.6).  One proves that if $ \H(A) = 0 $, then $ \T(A) = A $ and that $ \T \G(A) = \G(A)[-1]\{2\} $.  All these objects together form a spanning class and since we understand the behaviour of our functor on these objects, we can prove it is an equivalence.

In our much more general case, we will do something similar.  First we consider ``highest-weight'' objects; those objects $ A \in \D(\mu) $ which are killed by $ \E $.  Then we consider those objects in $ \D(\l) $ which are of the form $ \F^{(p)}(A) $ for some highest weight object $ A \in \D(\l+2p) $.  We will show that these objects together form a spanning class and that we can explicitly compute the action of $ \T $ on these objects.

More specifically, we will prove the following.
\begin{Theorem} \label{th:Tonhighweight}
Suppose $\l, p \ge 0$ and let $A \in \D(\l+2p)$ be a highest weight object. Then 
$$\T(\F^{(p)}(A)) \cong A[-p(\l+p)]\{p(\l+p+1)\}.$$
\end{Theorem}

\section{Preliminaries}

In this section we fix some notation and gather some facts and technical results we will use later on. 

\subsection{Notation}

We work over a fixed field $\k$. 

For any two objects $ A, B $ of a graded additive category, we write $ \Hom^i(A,B) $ for $ \Hom(A, B\langle i\rangle)$ and $ \Ext(A,B) $ for $ \oplus_i \Hom^i(A,B) $. We also write $\End^i(A)$ for $\Hom(A,A\langle i\rangle)$.  

We often have occasion to work with graded vector spaces (for example, $ \Ext(A, B) $ is graded).  If $ V = \oplus_i V_i $, then $ \dim_\bullet(V) $ will denote the graded dimension $ \sum_i \dim(V_i) q^i $.  On the other hand, $ \dim(V) $ will denote the total dimension, which is the evaluation at $q=1$ of the graded dimension.

Recall that $ [k] = \frac{q^k - q^{-k}}{q-q^{-1}} $ denotes a quantum integer, while
$\Bigl[ \begin{smallmatrix} j \\ k \end{smallmatrix} \Bigr] $ will denote the quantum binomial coefficient.

Because notation can get quite cumbersome we will have to abuse it at times. The biggest abuse is that we will often omit the grading shifts $\langle \cdot\rangle$ when they are not relevant. We might also omit the weights of certain functors if this information is redundant. This way, for example, the composition 
$$\iota: \F^{(r+1)}(\l) \rightarrow \F(\l+r) \circ \F^{(r)}(\l-1)\langle -r\rangle$$ 
might become $\iota: \F^{(r+1)} \rightarrow \F \F^{(r)}$. 

When we do want to emphasize the weights of functors,  it will be convenient to introduce the notation $ \F^s(\lambda) := \F(\lambda -s +1) \cdots \F(\lambda + s -1) : \D(\lambda + s) \rightarrow \D(\lambda - s)$.

We can compose the various morphisms $ \iota $ to obtain a morphism also denoted $ \iota : \F^{(s)}(\l) \rightarrow \F^s(\l)\langle -s(s-1)/2\rangle $.  Similarly, we have a morphism $ \pi : \F^s(\l) \rightarrow \F^{(s)}(\l)\langle -s(s-1)/2\rangle $.  

If one has a composition $\F^s$, then we can act by an $X$ on each $\F$ and by a $T$ on each pair of adjacent $\E$s. To simplify the indexing we will denote by $x_i$ the action of $X$ on the $i$th copy of $\F$ starting from the left. Similarly, we will denote by $t_i$ the action of $T$ on the $i$ and $i+1$st copies of $\F$ from the left. 

The $ x_1, \dots, x_s, t_1, \dots, t_{s-1} $ generate an action of the nil affine Hecke algebra on $ \F^s $.  The defining relations of the nil affine Hecke algebra are
\begin{equation*}
\begin{gathered}
x_i x_j = x_j x_i \quad x_i t_i = t_i x_{i+1} + 1 \quad x_i t_{i+1} = t_{i+1} x_i -1 \\
t_i x_j = x_j t_i, \text{ if } j \ne i, i+1 \quad t_i^2 = 0 \quad t_i t_{i+1} t_i = t_{i+1} t_i t_{i+1} \quad t_i t_j = t_j t_i, \text{ if } |i-j| \ge 2
\end{gathered}
\end{equation*}

For any element $ w \in S_s $, we can form a element $ t_w $ in the nil affine Hecke algebra by multiplying together the corresponding generators $ t_i $. In particular, we can form $ t_{w_0} $ using the longest element $w_0 \in S_s$. 

For $ \l \ge 0 $, an object $ A \in \D(\l) $ such that $ \E(A) = 0 $ is called a \emph{highest weight object}.

Let $\bG(k,n)$ denote the Grassmannian of $k$-planes in $\C^n$.  Denote by $H^\star(\bG(k,n))$ its symmetric cohomology, which means that we shift the usual grading on the cohomology of $\bG(k,n)$ so that it is symmetric with respect to degree zero. For example,
$$H^\star(\p^n) \cong \C \la n \ra \oplus \C \la n-2 \ra \oplus \dots \oplus \C \la -n+2 \ra \oplus \C \la -n \ra.$$
Note that if we keep track of the shift $\la \cdot \ra$ using the variable $q^\cdot$ then the graded dimension of $H^\star(\bG(k,n))$ is the quantum binomial coefficient $\Bigl[ \begin{smallmatrix} n \\ k \end{smallmatrix} \Bigr]$.

\subsection{Compatibility of certain maps}

There are some compatibilities between some of the morphisms occurring in the definition of strong $ \sl_2 $ categorification which follow easily from the other axioms. We collect two such compatibilities here for convenience.

\begin{Lemma} As morphisms $ \F^s \rightarrow \F^s $ the maps $t_{w_0}$ and $\iota \circ \pi$ are equal (up to a non-zero multiple).
\end{Lemma}
\begin{proof}
Both $t_{w_0}$ and $\iota \circ \pi$ are morphisms $\F^s \rightarrow \F^s \langle -s(s-1)/2\rangle$. Since $\F^s \cong \F^{(s)} \otimes H^\star({\bF}l(s))$, $\Hom^i(\F^{(s)},\F^{(s)}) = 0$ for $i < 0$ and $\End(\F^{(s)}) = \k \cdot \id$ this means that $t_{w_0}$ and $\iota \circ \pi$ must either be equal up to a non-zero multiple or equal zero. Since $\iota \circ \pi$ is a projection followed by an inclusion it is non-zero. Thus it remains to show that $t_{w_0} \ne 0$. 

We show this by induction on $s$. If $s=2$ (the base case) then $x_1t - tx_2 = 1$, so $t \ne 0$. Now assume $t_w \ne 0$ where $w \in S_{s-1}$ is the longest word. We need to show that $t_{w_0} \ne 0$ where $t_{w_0} = t_1 t_2 \dots t_{s-1} t_w$.

We can also do this by induction. Suppose we know that $t_{i+1} \dots t_{s-1} t_w \ne 0$ and assume that $t_i t_{i+1} \dots t_{s-1} t_w = 0$. Then 
\begin{equation*}
\begin{aligned}
0 = x_i t_i t_{i+1} \dots t_{s-1} t_w 
&= t_ix_{i+1} t_{i+1} \dots t_{s-1} t_w + t_{i+1} \dots t_{s-1} t_w \\
&= t_i t_{i+1} x_{i+2} t_{i+2} \dots t_{s-1} t_w + t_{i+1} \dots t_{s-1} t_w \\
\end{aligned}
\end{equation*}
where we use that $  t_i t_{i+2} \dots t_{s-1} t_w = t_{i+2} \cdots t_{s-1} t_i t_w = 0 $, since $ w$ is the longest word in $ S_{s-1} $.

Now we repeat this argument sliding the $x$ towards to right until we get 
$$0 = t_i \dots t_{s-1} x_s t_w + t_{i+1} \dots t_{s-1} t_w = t_i \dots t_{s-1} t_w x_s + t_{i+1} \dots t_{s-1} t_w$$
(here we used that $x_s$ commutes with $t_i$ for $i < s-1$ and hence with $t_w$). But $t_i \dots t_{s-1} t_w=0$ by assumption so we get $t_{i+1} \dots t_{s-1} t_w = 0$ (contradiction). Thus by induction  
$$t_{w_0} = t_1 \dots t_{s-1} t_w \ne 0$$
and we are done. 

\end{proof}

\begin{Lemma} \label{lem:pieta}
The following diagram commutes (up to a non-zero multiple) 
\begin{equation*}
\begin{CD}
\id @>>\eta> \E^{(s)}\F^{(s)} \\
@VV\eta V @VV\iota V \\
\E^s \F^s  @>\pi>> \E^s \F^{(s)} .
\end{CD}
\end{equation*}
\end{Lemma}
\begin{proof}
Let's consider the degree shifts more carefully. The map $\eta: \id \rightarrow \E^s \F^s(\l)$ is equal to the composition 
$$\id \xrightarrow{\eta} \E\F(\l+s-1)\langle \l+s-1\rangle \xrightarrow{\eta} \E\E\F(\l+s-3)\F\langle -2\l-2s+4\rangle \xrightarrow{\eta} \dots \xrightarrow{\eta} \E^s\F^s\langle- s\l\rangle.$$
So $\pi \eta: \id \rightarrow \E^s\F^{(s)}\langle -s\l - s(s-1)/2\rangle$. But now 
$$\E^s \cong \E^{(s)}\langle s(s-1)/2\rangle \oplus \bigoplus_i \E^{(s)}\langle k_i\rangle$$
where all the $k_i < s(s-1)/2$. This means that 
\begin{equation*}
\Hom(\id,\E^{(s)}(\l)\F^{(s)}\langle -s\l - s(s-1)/2+k_i\rangle) 
\cong \Hom(\F^{(s)}\la -s\l \ra, \F^{(s)}\langle -s\l - s(s-1)/2+k_i\rangle) 
\end{equation*}
which is zero since $k_i-s(s-1)/2 < 0$. 

So, for degree reasons, $\pi \eta$ factors through $\E^{(s)}\F^{(s)}\langle s(s-1)/2\rangle \rightarrow \E^s \F^{(s)}$. Since $\iota$ is the only such map (up to a multiple) this means $\pi \eta$ factors through $\iota$. Now the only map $\id \rightarrow \E^{(s)} \F^{(s)}\langle -s(s-1)/2\rangle$ is $\eta$ (up to a multiple). This means that $\pi \eta$ is equal to $\iota \eta$ up to a multiple. Since $\pi \eta$ is non-zero this multiple is non-zero and we are done.  
\end{proof}

\subsection{Krull-Schmidt and infinitesimal maps}\label{sec:infmaps}

Let $\D$ be a graded additive category over $\k$ which is idempotent complete (meaning that every idempotent splits). For example, the (derived) category of coherent sheaves on any variety is graded and idempotent complete. 

Suppose that (each graded piece of) the space of homs between two objects is finite dimensional. Then every object in $\D$ has a unique, up to isomorphism, direct sum decomposition into indecomposables (see section 2.2 of \cite{Rin}). In particular, this means that if $A,B,C \in {\mathcal C}$ then we have the
following cancellation laws:
\begin{eqnarray}\label{eq:A}A \oplus B \cong A \oplus C \Rightarrow B \cong C
\end{eqnarray}
\begin{eqnarray}\label{eq:B}A \otimes_\k \k^n \cong B \otimes_\k \k^n \Rightarrow A \cong B.
\end{eqnarray}

From now on suppose $ Q $ is an object with $\End^0(Q) = \k \cdot \id$ and $\End^i(Q) = 0$ for $i < 0$. Suppose $A, B$ are two objects in $\D$. We say that a map $f: A \rightarrow B$ is a {\bf $Q$-infinitesimal map} if for any map $\alpha: Q \la i \ra \rightarrow A$ and $\beta: B \rightarrow Q \la i \ra$ the composition $\beta \circ f \circ \alpha: Q \la i \ra \rightarrow Q \la i \ra$ is zero. We will let $\Inf_Q(A,B)$ denote the space of $Q$-infinitesimal maps $f: A \rightarrow \oplus_k B\langle k\rangle$. Sometimes we will write $\Inf(Q,A)$ to denote $\Inf_Q(Q,A)$. If every map $f: Q \la i \ra \rightarrow A$ is $Q$-infinitesimal then we say that {\bf $A$ contains no copies of $Q$}. 

Suppose $ f \in \Inf_Q(A,B) $ and $ g \in \Ext(B,C)$. It follows directly from the definition that $ gf \in \Inf_Q(A,C) $.  The same holds for composition on the other side. In particular $ \Inf_Q(A,A) $ is a two-sided ideal in $\Ext(A,A)$.  

\begin{Lemma} \label{lem:inf} Any object $A \in \D$ decomposes uniquely (up to permuting summands) as 
$$A \cong \bigoplus Q \la i \ra^{\oplus a_i} \oplus B$$
where $B$ contains no copies of $Q$. 
\end{Lemma}
\begin{proof}
If $\alpha: Q \la i \ra \rightarrow A$ is not $Q$-infinitesimal then we can find a map $\beta: A \rightarrow Q \la i \ra$ so that the composition $\beta \circ \alpha: Q \la i \ra \rightarrow  Q \la i \ra$ is the identity (this is where we use that $\End^0(Q) = \k \cdot \id$). This means that $ \alpha \beta \alpha \beta = \alpha \beta $. Hence $ P = \alpha \beta $ is an idempotent and by idempotent completeness it is a projection onto a direct summand which is isomorphic to $ Q \la i \ra$. Thus $A \cong Q \la i \ra \oplus A'$. Repeating this way we get that $A \cong \bigoplus Q \la i \ra^{\oplus a_i} \oplus B$ where $B$ contains no copies of $Q$. This decomposition is unique because every direct sum decomposition is unique (up to permuting summands). 
\end{proof}

In the decomposition $A \cong \bigoplus Q \la i \ra^{\oplus a_i} \oplus B$ of Lemma \ref{lem:inf} we call $\bigoplus Q \la i \ra^{\oplus a_i}$ the $Q$-part of $A$. Also, we let
$$ \dim_Q(A) := \dim(\Ext(Q,A)/\Inf(Q,A)) = \sum_i a_i.$$
The second equality follows from the fact that $ \dim_Q(Q\la i \ra) = 1 $ for all $ i $.

\begin{Lemma} \label{lem:infzero} Suppose $f : \bigoplus Q \la i \ra^{\oplus a_i} \rightarrow \bigoplus Q \la i \ra^{\oplus a_i}$ is an isomorphism. Then $f + \delta$ is an isomorphism for any infinitesimal map $ \delta $.
\end{Lemma}
\begin{proof}
Since there are no negative degree maps between the $Q$'s the matrix representing $f$ is upper triangular. Since $\End(Q) = \k \cdot \id$ the diagonal entries are matrices with values in $\k$. Since $f$ is invertible each of these matrices is invertible. Now if $\delta$ is infinitesimal then it is strictly upper triangular so that the diagonal of $f + \delta$ still contains invertible matrices. Thus $f + \delta$ is invertible. 
\end{proof}

Often in this paper we will consider an object $ A $ together with a subspace $ V \subset \Ext(Q, A) $, and we will want to know if the canonical map $ Q \otimes V \rightarrow A $ is an isomorphism onto the $ Q $ part of $ A$.  
This question is studied in the following Lemma, which we use frequently. 

\begin{Lemma} \label{th:modinf}
Suppose that $ V \subset \Ext(Q, A) $ is a subspace.  

If $ Q \otimes V \rightarrow A $ is an isomorphism onto the $ Q $ part of $ A$, then the following hold
\begin{enumerate}
\item $ [V] = \Ext(Q,A)/\Inf(Q,A) $.
\item $ \dim V = \dim_Q(A) $.
\item No non-zero elements of $ V $ are $ Q $-infinitesimal.
\end{enumerate}
Conversely, if any two of the above three things hold, then the third holds and $ Q \otimes V \rightarrow A $ is an isomorphism onto the $ Q $ summand of $ A $.
\end{Lemma}
\begin{proof}
This follows immediately from the definition of $Q$-infinitesimal and Lemma \ref{lem:infzero}. 
\end{proof}

\subsection{Convolutions of complexes} \label{se:conv}

Now suppose $\D$ is also a triangulated category with cohomological shift $[\cdot]$.  We recall the concept of convolution in $\D$ (see \cite{GM} section IV, exercise 1). The idea is to generalize the Cone of a morphism to the case of a complex with more than one morphism (or equivalently with more than two objects). 

Let $ (A_\bullet, f_\bullet) = A_n \xrightarrow{f_n} A_{n-1} \rightarrow \cdots \xrightarrow{f_1} A_0 $ be a sequence of objects and morphisms in $ \mathcal{D}$ such that $ f_i \circ f_{i+1} = 0 $. Such a sequence is called a \textbf{complex}.

A \textbf{(right) convolution} of a complex $ (A_\bullet, f_\bullet) $ is any object $ B $ such that there exist
\begin{enumerate}
\item objects $ A_0 = B_0, B_1, \dots, B_{n-1}, B_n = B $ and
\item morphisms $ g_i : B_i [-i] \rightarrow A_i $, $ h_i : A_i \rightarrow B_{i-1}[-(i-1)] $ (with $ h_0 = id $)
\end{enumerate}
such that
\begin{equation} \label{eq:distB}
B_i[-i] \xrightarrow{g_i} A_i \xrightarrow{h_i} B_{i-1}[-(i-1)]
\end{equation}
is a distinguished triangle for each $ i $ and $ g_{i-1} \circ h_{i} = f_i $. Such a collection of data is called a \textbf{Postnikov system}. Notice that in a Postnikov system we also have $f_{i+1} \circ g_i = (g_{i+1} \circ h_i) \circ g_i = 0 $ since $h_i \circ g_i = 0$. 
The following result is a sharper version of Lemma 1.5 from \cite{O}.

\begin{Proposition} \label{th:uniquecone}
Let $ (A_\bullet, f_\bullet) $ be a complex. The following existence and uniqueness results hold.
\begin{enumerate}
\item If $\Hom(A_{i+k+1} [k], A_i) = 0 $ for all $ i \ge 0, k \ge 1 $, then any two convolutions of $ (A_\bullet, f_\bullet) $ are isomorphic.
\item If $ \Hom(A_{i+k+2} [k], A_i) = 0 $ for all $ i \ge 0, k \ge 1 $, then $ (A_\bullet, f_\bullet) $ has a convolution.
\end{enumerate}
\end{Proposition}
\begin{proof}
See Proposition 8.3 of \cite{ck1}. 
\end{proof}

Whenever both conditions of Proposition \ref{th:uniquecone} are satisfied we will denote the convolution by $\Cone(f)$. The reason for this is that when the complex $A_\bullet$ has only two terms then the two conditions are automatically satisfied and the convolution is just the usual cone. 

\subsection{Exact complexes} \label{sec:exactcomplex}
Fix an object $Q \in \D$, without any conditions on $\End^*(A)$ (in particular we do not require the assumptions of section \ref{sec:infmaps} ).

Consider a complex $(A_\bullet, f_\bullet)$ as above where for each $ i$, we have a fixed isomorphism $A_i = Q \otimes V_i$ for some graded vector space $ V_i = \oplus_j V_{ij}$, so that 
$ A_i = \oplus_j Q \la j \ra \otimes V_{ij}$.
We insist that the $f_i$'s contain no negative degree maps, so that the matrix entries
$$ Q \la j \ra  \otimes V_{ij} \rightarrow Q \la j' \ra \otimes V_{i-1 j'}$$
of $f_i$ are zero if $j' < j$. We also require that the degree zero maps are multiples of the identity, so that the matrix entries
$$Q \la j \ra \otimes V_{ij} \rightarrow Q \la j \ra \otimes V_{i-1 j}$$
of $f_i$ are of the form $I \otimes M_{ij}$ for some linear maps $M_{ij} : V_{ij} \rightarrow V_{i-1 j}$.  Note that these notions depend on the particular isomorphisms $ A_i = Q \otimes V_i $ which we have fixed.

Because there are no negative degree maps the fact that $f_i \circ f_{i+1} = 0$ means that $M_{ij} \circ M_{i+1 j} = 0$ for all $j$.  Thus from $(A_\bullet, f_\bullet)$ together with fixed isomorphisms as above, we get a complex of graded vector spaces $ (V_\bullet, M_\bullet) $.

We say that $(A_\bullet, f_\bullet)$ is {\bf exact} if $(V_\bullet, M_\bullet) $ is exact as a complex of graded vector spaces.  

\begin{Lemma}\label{lem:conezero}
If $(A_\bullet, f_\bullet)$ is an exact complex then it has a unique right convolution which is zero.  More precisely, every partial Postnikov system $ B_i, \dots, B_0, g_i, \dots, g_0, f_i, \dots, f_0 $ extends to a full Postnikov system with $ B_n = 0 $.
\end{Lemma}
\begin{proof}
We cannot apply Proposition \ref{th:uniquecone} directly, so instead we proceed by induction on the length of the complex $n$. The base case $ n=0 $ follows immediately.

Now consider a general $n$. Let's look at the map $f_1: A_1 \rightarrow A_0$.  The corresponding map $ M_1 : V_1 \rightarrow V_0 $ is surjective, so we can split $ V_1 = W_1 \oplus U_1 $ where $ W_1 = \ker(M_1) $ and the restriction $ M_1 : U_1 \rightarrow V_0 $ is an isomorphism.  This gives us a splitting 
\begin{equation} \label{eq:mainsplit}
A_1 = Q \otimes W_1 \oplus Q \otimes U_1 
\end{equation}
With respect to this direct sum decomposition, let us write $ f_1 = [r \ s]$.

Consider the matrix entries of $ s$ with respect to the direct sum decompositions
\begin{equation*}
Q \otimes U_1 = \oplus_j Q\la j \ra \otimes U_{1j} \quad A_0 = \oplus_j Q \la j \ra \otimes V_{0j} 
\end{equation*}
By assumption, this will be an upper triangular matrix with diagonal entries given by $ I \otimes M_{1j}$.  These diagonal entries are isomorphisms, so by upper triangularity, $ s : Q\otimes U_1 \rightarrow A_0 $ is an isomorphism.

Hence we can define $ B_1 = Q \otimes W_1 $ and define $ g_1 $ by $\left[ \begin{smallmatrix} I \\ -rs^{-1} \end{smallmatrix} \right] $ as in Lemma \ref{th:simpletri}.

With respect to the direct sum decomposition (\ref{eq:mainsplit}), let us write $ f_2 $ as $ \left[ \begin{smallmatrix} p \\ q \end{smallmatrix} \right] $.  We can define $ h_2 := p$.  It is immediate that $ f_2 = g_1 \circ h_2 $.

Now consider the new complex $ A_n \xrightarrow{f_n} \dots \xrightarrow{f_3} A_2 \xrightarrow{h_2} B_1 $.  We claim that this complex is exact.  Since $ h_2 $ is a matrix entry of $ f_2 $ with respect to (\ref{eq:mainsplit}), we see that it has no negative degree entries.  Note that $ M_2 : V_2 \rightarrow V_1 = W_1 \oplus U_1 $ has image $ W_1 = \ker(M_1) $ by the exactness of the complex $ M_\bullet $.  So we see that the diagonal entries of $ h_2 $ are exactly $ I \otimes M_{2j} $.  The resulting complex of vector spaces
\begin{equation*}
V_n \xrightarrow{M_n} \cdots \xrightarrow{M_3} V_2 \xrightarrow{M_2} W_1 
\end{equation*}
is exact.  Hence we are in the setup of this Lemma, but with a complex with fewer terms.  Hence by induction, the complex $ A_\bullet $ has a Postnikov system with convolution 0.

To show uniqueness, just note that in any Postnikov system $ B_1 $ must be isomorphic to $Q \otimes W_1$.  Moreover, $ g_1 $ is the inclusion of a direct summand.  Thus there is no choice for $ h_2 $ and the resulting complex $ A_n \xrightarrow{f_n} \dots \xrightarrow{f_3} A_2 \xrightarrow{h_2} B_1 $ must satisfy the hypotheses of this Lemma.
\end{proof}

\begin{Lemma} \label{th:simpletri}
Let $ B, C, D $ be objects in a triangulated category with maps $ r : B \rightarrow D $ and $ s : C \rightarrow D $, such that $ s $ is an isomorphism.  Then there is a distinguished triangle
\begin{equation*}
B \xrightarrow{\left[ \begin{smallmatrix} I \\ -s^{-1}r \end{smallmatrix} \right]} B \oplus C \xrightarrow{[r \ s]} D
\end{equation*}
\end{Lemma}

\begin{proof}
The given triangle $ B \rightarrow B \oplus C \rightarrow D $ is isomorphic to the triangle \begin{equation} \label{eq:newtri}
 B \xrightarrow{\left[ \begin{smallmatrix} I \\ 0 \end{smallmatrix} \right]} B \oplus C \xrightarrow{[0 \ I]} C
\end{equation}
via the vertical maps $ I, \left[ \begin{smallmatrix} 1 \ 0 \\ s^{-1}r \ I \end{smallmatrix} \right], s^{-1} $.

The triangle (\ref{eq:newtri}) is distinguished, hence so is the original triangle.
\end{proof}

\section{Computation of the complex on images of highest weight objects}

The purpose of this section is to prove the following result which will be the key tool used to prove Theorem \ref{th:Tonhighweight}.

\begin{Theorem} \label{th:Tonhighweight2}
Suppose $\l,p \ge 0$ and let $A \in \D(\l+2p)$ be a highest weight object (i.e. $\E(A)=0$). Then there exists some $v_p \in \End^{2p(\l+p)}(\F^{(\l+p)} \E^{(p)} \F^{(p)})$ such that 
$$\F^{(\l+p)}(A) \la -p(\l+p+1) \ra \xrightarrow{\gamma_p} \Theta_p \F^{(p)}(A) \xrightarrow{d_p} \Theta_{p-1} \F^{(p)}(A) \xrightarrow{d_{p-1}} \dots \xrightarrow{d_1} \Theta_0 \F^{(p)}(A)$$
is an exact complex in the sense of section \ref{sec:exactcomplex} and Lemma \ref{lem:conezero}. The map $\gamma_p$ is induced by the composition
\begin{eqnarray*}
\F^{(\l+p)}(p) \la -p(\l+p+1) \ra &\xrightarrow{\eta}& \F^{(\l+p)} \E^{(p)} \F^{(p)}(\l+p) \la -p(\l+p+1)-p(\l+p) \ra \\
&\xrightarrow{v_p}& \F^{(\l+p)} \E^{(p)} \F^{(p)} \la -p \ra \cong \Theta_p \F^{(p)}.
\end{eqnarray*}
\end{Theorem}

In the remainder of this section we prove Theorem \ref{th:Tonhighweight2}. To do this we have to analyze the terms in the complex $\Theta_\bullet \F^{(p)} $ carefully. A particular term in this complex is $$ \Theta_s \F^{(p)} = \F^{(\l+s)}(s) \E^{(s)}(\l+s) \F^{(p)}(\l+p)\la-s\ra. $$  As we will see below this is isomorphic to a direct sum of shifts of copies of $ \F^{(\l + p)} $ direct sum with $ \U $ where $ \U $ will always denote some functor which is a direct sum of terms $ \F^{(a)} \E^{(b)} $ for $ b \ne 0 $. We will not need to know the precise form of $\U$ because $\U(A)=0$ for any highest weight object $A$. 

Using the data on the strong categorical $\sl_2 $ action, we construct a particular isomorphism onto the $ \F^{(\l+p)} $ part of $ \Theta_s \F^{(p)} $ for each $ s $.  Then we use these isomorphisms in order to compute $\Theta_* \F^{(p)}$. 

\subsection{Abstract isomorphisms}

We begin by examining abstractly the pieces in the complex $\Theta$.

\begin{Lemma} \label{th:abstractiso}
If $ \l-2a\ge 0 $ then there exists an isomorphism
\begin{equation*}
\E^{(b)}(\l-2a+b) \F^{(a)}(\l-a) \cong \bigoplus_{j = 0}^b \F^{(a-j)} \E^{(b-j)} \otimes_{\k} H^\star(\bG(j, \l-a+b)).
\end{equation*}
\end{Lemma}

\begin{Remark} The isomorphism in Lemma \ref{th:abstractiso} actually holds whenever $\l-a+b \ge 0$. Since the proof of this more general statement is tedious we only prove the weaker statement (which satisfies our needs in the rest of the paper).  
\end{Remark}

\begin{proof}
Using the categorified commutation relation, we find that 
$$ \E^b \F^a = \E^{b-1} \E \F \F^{a-1} \cong \E^{b-1} \F \E \F^{a-1} \oplus \E^{b-1} \F^{a-1} \oplus \cdots \oplus \E^{b-1} \F^{a-1}.$$  Continuing in this way we deduce an isomorphism between $ \E^b \F^a $ and a direct sum of terms of the form $ \E^{a-j} \F^{b-j} $ for some $ j$.  

Now, we know that $ \E^b \cong {\E^{(b)}}^{\oplus b!} $ and $ \F^a \cong {\F^{(a)}}^{\oplus a!} $.  So by uniqueness of direct sum decompositions, we see that
\begin{equation*}
\E^{(b)}(\l-2a+b) \F^{(a)}(\l-a) \cong \bigoplus_{j = 0}^b \F^{(a-j)} \E^{(b-j)} \otimes_{\k} V_j
\end{equation*}
for some graded vector spaces $ V_j $.

The graded dimension of the vector spaces $ V_j $ may be determined using the ordinary representation theory of $ U_q(\sl_2) $.  From \cite{Lu}, 23.1.3 we see that if $ v $ is a vector in a representation of $ U_q(\sl_2) $ and $ v $ has weight $ \l$, then
\begin{equation*}
e^{(b)} f^{(a)} v = \sum_j \Bigl[ \begin{smallmatrix} \l - a+ b \\ j \end{smallmatrix} \Bigr] f^{(a-j)} e^{(b-j)} v
\end{equation*}
From this we deduce that $ \dim_\bullet(V_j) = \Bigl[ \begin{smallmatrix}\l - a + b \\ j \end{smallmatrix} \Bigr] $.  Hence the result follows.
\end{proof}

Combining this result with the hypothesis about composition of $ \F$s, we obtain the following corollary.
\begin{Corollary} \label{th:abstractiso2}
There exists an isomorphism
\begin{equation*}
\Theta_s(\F^{(p)}(\l +p)) \cong \bigoplus_{j \ge 0} \F^{(\l +s +p-j)}(p+s-j) \E^{(s-j)}(\l +2p +s -j) \otimes_{\k} V_j \la -s \ra
\end{equation*}
where $V_j$ is the graded vector space $ V_j \cong H^\star(\bG(j, \l+p - s)) \otimes H^\star(\bG(p-j, \l+s +p-j)) $.

\end{Corollary}

We will now concentrate only on the $ \F^{(\l +p)} $ part of $ \Theta_s \F^{(p)} $.  This is precisely the last direct summand appearing above, by the following lemma.

\begin{Lemma} \label{th:notdirectsummand}
If $ \l + r+a > 0 $, then $ \F^{(r+ a)}(\l +a + r) \E^{(a)}(\l+a+2r) $ contains no copies of $ \F^{(r)}(\l + r)$. 
\end{Lemma}

\begin{proof}
It suffices to show that at least one of the $\Hom$ spaces
\begin{equation*} 
\Hom(\F^{(r)}(\l + r)\langle k\rangle, \F^{(r+a)}(\l + a+ r) \E^{(a)}(\l+a+2r)) ,\
\Hom(\F^{(r+a)}(\l +a +r) \E^{(a)}(\l+a+2r), \F^{(r)}(\l + r)\langle k\rangle) 
\end{equation*}
is equal to 0.
For assume that $\Hom(\F^{(r)}\langle k\rangle, \F^{(r+a)} \E^{(a)}(\l+a+2r)) $ is non-zero. By adjunction we have 
\begin{equation*}
\Hom(\F^{(r)}\langle k\rangle, \F^{(r+a)} \E^{(a)}(\l+a+2r)) = \Hom(\F^{(r)} \F^{(a)}\langle k+a(\l+a+2r)\rangle, \F^{(a+r)})
\end{equation*}
Then by composition of $\F$s, we have 
\begin{equation*}
\F^{(r)} \F^{(a)} \cong \F^{(r+a)} \otimes H^\star \bG(a, a+r) = \F^{(r+a)}\langle ar\rangle \oplus \cdots \oplus \F^{(r+a)}\langle -ar\rangle 
\end{equation*}
By assumption the Hom space is non-zero. Since there are no negative degree endomorphisms of $ \F^{(a+r)}$, we see that \begin{equation*}
k+a(\l +a+2r) - ar \le 0 \Leftrightarrow k \le -a(\lambda + r +a).
\end{equation*}

Repeating the same analysis under the assumption that $ \Hom(\F^{(r+a)} \E^{(a)}(\l-a), \F^{(r)}\langle k\rangle) $ is non-zero, we see that
\begin{equation*}
k \ge a (\lambda + r +a).
\end{equation*}

Thus if both $\Hom$ spaces are non-zero and $ \lambda + r + a > 0 $, we reach a contradiction.
\end{proof}

Our goal now is to find a specific subspace of $ \Ext(\F^{(\l + p)}, \Theta_s \F^{(p)}) $ which maps isomorphically onto the $ \F^{(\l+p)} $ part.

\subsection{Rephrasing of the definition}

We start by rephrasing the statement of some isomorphisms from the definition of a strong categorical $ \sl_2 $ action.

\begin{Proposition} \label{th:decomp2}
We have an isomorphism $ \F^{(r)} \otimes V \iota \rightarrow \F^{(r-1)}\F $ where $ V \iota \subset  \Ext(\F^{(r)}, \F^{(r-1)}\F) $ is the subspace obtained by composing the subspace $ V = \Span\{1,x, \dots, x^{r-1}\} \subset \Ext(\F,\F) $ with the map $\iota$. To summarize, the isomorphism is given by
\begin{equation*}
\F^{(r)} \xrightarrow{\iota} \F^{(r-1)}\F \langle -(r-1)\rangle \xrightarrow{V} \F^{(r-1)}\F
\end{equation*}
\end{Proposition}

\begin{Proposition} \label{th:decomp1}
Assume that $ \lambda \ge 0$.
We have an isomorphism $ \F(\lambda -1)\E(\lambda -1) \oplus \id \otimes V \eta \rightarrow \E(\lambda+1)\F(\lambda+1) $ where the map on the first factor is given by 
$$ \F\E \xrightarrow{\eta} \E\F\F\E\langle -\lambda+1\rangle \xrightarrow{t} \E\F\F\E\langle -\lambda-1\rangle \xrightarrow{\varepsilon} \E\F $$ 
and where $ V \eta \subset \Ext(\id, \E\F) $ which is obtained by composing the subspace $ V = \Span\{1, x, \dots, x^{\lambda-1}\} $ of $ \Ext(\F,\F) $ with the map $ \eta$: 
\begin{equation*}
\id \xrightarrow{\eta} \E\F\langle -\lambda+1\rangle \xrightarrow{V} \E\F
\end{equation*}
\end{Proposition}

\subsection{Decomposition of $ \F^{(r)} \F^{(s)}$}

Here is the first of the isomorphisms that we will need.

\begin{Proposition} \label{th:firstiso}
The map $ \pi V \iota \otimes \F^{(s+r)} \rightarrow \F^{(s)} \F^{(r)} $ is an isomorphism. Here $ \pi V \iota $ is the space of maps given by
\begin{equation*}
\F^{(s+r)} \xrightarrow{\iota} \F^{(s)}\F^r \xrightarrow{V} \F^{(s)}\F^r \xrightarrow{\pi} \F^{(s)}\F^{(r)}
\end{equation*}
where $ V \subset \Ext(\F^r, \F^r) $ is the span of $ x_1^{a_1} \cdots x_r^{a_r} $ for $ 0 \le a_1 < \dots < a_r < r+s $. 
\end{Proposition}
\begin{proof}
We start with the isomorphism $ \F^{(s+r)} \otimes V' \iota \rightarrow \F^{(s)}\F^r $ which comes from iterating Proposition \ref{th:decomp2}.  Here $ V' $ is the span of $ x_1^{t_1} \cdots x_r^{t_r} $ for $ 0 \le t_i < i+s $.  Now we would like to compute $ \pi V' \iota $.

By Corollary \ref{th:sym2c}, we see that if $ f \in V' $ is symmetric in any two variables, then $ \pi f \iota = 0 $.  From this it is immediate that $ \pi V \iota = \pi V' \iota $.  

Since $[ V' \iota ] = \Ext(\F^{(s+r)}, \F^{(s)}\F^r)/\Inf $, we see that $[ \pi V \iota ] = [ \pi V' \iota ] = \Ext(\F^{(s+r)}, \F^{(s)}\F^{(r)})/\Inf $.  

Note that $ (\F^{(s)}\F^{(r)})^{\oplus r!} \cong \F^{(s)}\F^r$.  Hence $ \dim_{\F^{(s+r)}}(\F^{(s)}\F^{(r)}) = \frac{(r+s)!}{s!}/r! $.  On the other hand, $\dim V = \frac{(r+s)!}{s!}/r!$ and $\dim \pi V i \le \dim V$, so by Lemma \ref{th:modinf}, the result follows. 
\end{proof}

\begin{Lemma} \label{th:sym2}
If $ f $ is a polynomial in $ x_1, \dots, x_s $ which is symmetric in $ x_i, x_{i+1} $, then 
\begin{enumerate}
\item $t_i f = f t_i $, 
\item $ t_i f t_i = 0 $, and
\item $t_{w_0} f t_{w_0} = 0 $.
\end{enumerate}
\end{Lemma}

\begin{proof}
Any polynomial which is symmetric in $ x_i, x_{i+1} $ can be written as a sum of terms of the form $ g (x_i + x_{i+1})^k (x_1 x_2)^l $ where $ g $ is a polynomial in $ x_1, \dots, x_{i-1}, x_{i+2}, \dots, x_n $.  Hence to prove (i), it suffices to check that $ t_i g = g t_i $, $ t_i(x_i + x_{i+1}) = (x_i + x_{i+1}) t_i $ and $ t_i x_i x_{i+1} = x_i x_{i+1} t_i $.  But these are all easy consequence of the nil affine Hecke relations.

(ii) follows from (i) and the fact that $ t_i^2 = 0$.

For (iii), pick one reduced word for $w_0 $ that begins with $i$ and another which ends with $i$.  Thus $ t_{w_0} = t_w t_i $ and $ t_{w_0} = t_i t_{w'}$.  Hence (iii) follows from (ii).
\end{proof}

\begin{Corollary} \label{th:sym2c}
If $ f $ is any polynomial in $ x_1, \dots, x_r $ which is symmetric in any two adjacent variables, then the composition $ \pi \circ f \circ \iota $ is zero.
\begin{equation*}
\F^{(r)} \xrightarrow{\iota} \F^r \xrightarrow{f} \F^r \xrightarrow{\pi} \F^{(r)}.
\end{equation*}
\end{Corollary}
\begin{proof}
It suffices to show that $ \iota \circ \pi \circ f \circ \iota \circ \pi  = 0$   But $ \iota \circ \pi = t_{w_0} $, so it suffices to show that $ t_{w_0} f t_{w_0} = 0 $ which follows from Lemma \ref{th:sym2} part (iii).
\end{proof}

\subsection{Decomposition of $ \E^{(s)}\F^{(s+r)}$}

In this section we will analyse the $ \F^{(r)}(\l) $ part of $ \E^{(s)}(\l-r-s) \F^{(s+r)}(\l -s) $ under the assumption that $ \l - 2s - r \ge 0 $.  This will turn out to be a bit difficult, so we will do it in three steps.  First, we will consider the $ I $ part of $ \E^s\F^{(s)} $.  Then we will consider the $ \F^{(r)} $ part of $ \E^s \F^{(s+r)} $.  Finally, we will consider the $ \F^{(r)} $ part of $ \E^{(s)} \F^{(s+r)}$.

\subsubsection{}

\begin{Lemma} \label{th:simpleiso}
If $\l \ge 2s \ge 0$ then there is an isomorphism
\begin{equation*}
\id(\lambda) \otimes W_1 \eta \oplus \U \rightarrow \E^s(\lambda - s) \F^s(\lambda - s)
\end{equation*}
where $ \eta : \id \rightarrow \E^s \F^s $ and $W_1 \subset \Ext(\F^s \F^s, \F^s \F^s) $ is a subspace spanned by $ t$s and $ x$s.
\end{Lemma}

Here $ \U $ is a direct sum of $ \F^{(a)}\E^{(a)}$ with $ a > 0 $.  It is given in Lemma \ref{th:abstractiso}.  We do not care about the precise description of the space $ W_1$, though it is possible to write down an explicit basis of monomials.

The proof of this Lemma is inspired by the pictorial diagrams of Lauda \cite{l1}.

\begin{proof}
We will ignore shifts in this proof for ease of notation.

By repeatedly applying Proposition \ref{th:decomp1}, we see that there is an isomorphism
$ \oplus \id \oplus \U \stackrel{\sim}{\rightarrow} \E^s \F^s $ where on each direct summand the map is given by a composition of \begin{equation*} \label{eq:building}
 \eta : \id \rightarrow \E\F,\ \varepsilon : \F\E \rightarrow \id, \ T : \F\F \rightarrow \F\F, \ X : \F \rightarrow \F 
\end{equation*} 
in particular, we are only using the adjuction $ \E = \F_R $.

So it suffices to show that any map $ \id \rightarrow \E^s \F^s $ which can be written as a composition of the maps in (\ref{eq:building}) can actually be rewritten as $ \id \xrightarrow{\eta} \E^s \F^s \xrightarrow{f} \E^s \F^s $ where $ f $ is a composition of $ X$s and $T$s acting on $ \F^s $.

Let $ g : \id \rightarrow \E^s\F^s $ be one of the above maps built as a composition of (\ref{eq:building}).  Then the map $ h = \varepsilon I \circ I g  : \F^s \rightarrow \F^s $
\begin{equation*}
\F^s \xrightarrow{I g} \F^s \E^s \F^s \xrightarrow{\varepsilon I} \F^s 
\end{equation*}
satisfies the hypotheses of Lemma \ref{th:straighten}.  Hence it can be written as a composition of $ X$s and $T$s.

Now, we have that our original map $ g $ can be rewritten as 
\begin{equation*}
\id \xrightarrow{g} \E^s\F^s \xrightarrow{\eta I} \E^s \F^s \E^s \F^s \xrightarrow{ I \varepsilon I} \E^s \F^s
\end{equation*}
which in turn equals the composition
\begin{equation*}
\id \xrightarrow{\eta} \E^s \F^s \xrightarrow{I h} \E^s \F^s
\end{equation*}
Since $ h $ is a composition of $ X$s and $T$s, the result follows.
\end{proof}

\begin{Lemma} \label{th:straighten}
Let $ (\L,\R) $ be a pair of adjoint functors.  Let $ \eta: \id \rightarrow \R \L, \varepsilon: \L \R \rightarrow \id $ be the unit and co-unit of adjunction.  Let $ \{f_1, \dots, f_k \} $ be a set of morphisms with each $ f_i \in \End(\L^j) $ for some $j $.

Then any map $ \L^s \rightarrow \L^s $ which is the composition of $ \eta, \varepsilon, \{f_i\} $ can be expressed as a composition of the $ \{f_i \} $ alone.
\end{Lemma}
\begin{proof}
Note that the number of $ \eta $ and $ \varepsilon $ in the composition must be the same.  Hence it suffices to show that we can rewrite our composition as a composition with one less $ \eta $ and $ \varepsilon $.

Fix any $ \eta $ in the composition.  This creates an $ \R $ which is then killed by a later $ \varepsilon $.  In between, all the morphisms $g$ in the composition do not affect $ \R $.
\begin{equation*}
\cdots \xrightarrow{\eta} \ldots \R \L \ldots \xrightarrow{g_1} \cdots \xrightarrow{g_2} \ldots \L \R \ldots \xrightarrow{\varepsilon} \cdots
\end{equation*}
Hence each $ g $ must be either of the form $ hI_{\R} $ or $ I_{\R}h $.  These two sorts of morphisms commute.  The first type also commute with the $ \eta $ which created the $ R $ and the second type commute with the $ \varepsilon $ with kills the $ \R $.

So we can rewrite our composition so that the $ \eta $ and the $\varepsilon $ are adjacent.  Since their composition $\L \xrightarrow{I \eta} \L\R\L \xrightarrow{\varepsilon I} \L$ is the identity, we are done.
\end{proof}

\begin{Proposition} \label{th:secondiso1}
If $ \l -2s \ge 0 $, then the following map is an isomorphism 
\begin{equation*}
\id(\lambda) \otimes \pi W_2 \eta  \oplus \U \rightarrow \E^s \F^{(s)}(\l-s)
\end{equation*}
where $ W_2 $ is the subspace of $ \Ext(\E^s \F^s, \E^s \F^s) $ spanned by the monomials $ x_1^{a_1} \cdots x_s^{a_s} $ (acting on the $\F^s$ term) such that $ a_i < \lambda-s+i$. To summarize, the isomorphism is given by maps
\begin{equation*}
\id \xrightarrow{\eta} \E^s \F^s \xrightarrow{W_2} \E^s \F^s \xrightarrow{\pi} \E^s \F^{(s)}.
\end{equation*}
\end{Proposition}

\begin{proof}
We start by applying Lemma \ref{th:simpleiso}, to obtain the isomorphism $ \id \otimes W_1 \eta \oplus \U \rightarrow \E^s \F^s$.  

Note that $ \E^s \F^{(s)} $ is a direct sum of shifts of $\id $ and $ \U $ by Lemma \ref{th:abstractiso}.  By Lemma \ref{th:notdirectsummand}, $ \id $ is not a direct summand of $ \U $.

Now, we want to compute $ \pi W_1 \eta $ where $ \pi $ is the projection $ \E^s \F^s \rightarrow \E^s\F^{(s)}\langle s(s-1)/2\rangle $.  Any basis vector in $ W_1 $ as above can be rewritten using the nil affine Hecke relations as a sum of terms with all $ t_i $ on the LHS.  However any terms with any $ t $ on left will become 0 when composed with $ \pi $ (since $ \pi t_i = 0 $ for all $i $).   Hence we see that $ \pi W_1 \eta$ is spanned by polynomials in the $x$s.  

By Lemma \ref{th:modinf}, we see that $ \Ext(\id, \E^s \F^s)/\Inf = [W_1 \eta] $.  Hence $ \Ext(\id, \E^s \F^{(s)})/\Inf = [\pi W_1 \eta] $.

Let $ W_2 $ be the subspace of $ \Ext(\F^s, \F^s) $ spanned which is spanned by monomials of the form $ x_1^{a_1} \cdots x_s^{a_s} $ where $ 0 \le a_i < \lambda- s+i $. Combining Corollary \ref{th:linearcomblem} below and the above observations, we see that 
$$ [ \pi W_2 \eta] = [\pi \k[x_1, \dots, x_s] \eta] \supset [\pi W_1 \eta] = \Ext(\id, \E^s \F^{(s)})/\Inf.$$  

To complete the proof, note that $ \dim(\pi W_2 \eta) \le \dim W_2 \le \lambda!/(\lambda - s)!$.  

Also, $ \dim_{\id}(\E^s \F^s) = \lambda! s! / (\l+s)! $.  This can be see by considering an irreducible representation of highest weight $ \l $ and applying $ \E^s \F^s $ to this highest weight vector.  

Since,  ${\F^{(s)}}^{\oplus s!} = \F^s $, we obtain that $ \dim_{\id}(\E^s \F^{(s)}) = (\lambda)!/(\lambda s)!$. Hence by Lemma \ref{th:modinf} the result follows.
\end{proof}

\begin{Lemma} \label{th:linearcomblem}
Suppose $\l \ge 2s$ and let  $ i \in \{ 1, \dots, s \} $.  Let $ W_2^i $ be the subspace of  $ \Ext(\E^s \F^s, \E^s \F^{s}(\l-s))$ spanned by $ x_i^{a_i} \cdots x_s^{a_s} $ where $ a_j < \l -s +j $.  Then
\begin{equation*}
[ \pi x_i^{\l-s+i}] \in [ \pi W_2^i]  \subset \Ext(\E^s \F^s, \E^s \F^{(s)}(\l-s))/\Inf.
\end{equation*}
\end{Lemma}

\begin{proof}
Notice that it suffices to show that for $w_0 \in S_s$
$$[ t_{w_0} x_i^{\l -s +i}] \in [ t_{w_0} W_2^i] \subset \Ext(\E^s \F^s, \E^s \F^s(\l-s))/\Inf.$$ 

We start with $i=s$. Now $\E^s \F^s = \E^{s} \F^{s-1} \F(\l-1)$ and $x_s^{\l-s+i}$ acts by $X^{\l-s+i}$ on the right-most term $\F(\l+1)$. Rewriting $\E^{s} \F^{s-1} \cong \bigoplus_{i \ge 0} \F^{(i)} \E^{(i+1)} \otimes V_i$ for some vector spaces $V_i$ we see that $x_s^b$ acts by $X^b$ on the right-most factor in 
$$\E^s \F^s(\l+s) \cong \bigoplus_{i \ge 0} \F^{(i)} \E^{(i+1)} \F(\l+1) \otimes V_i.$$
Now, for $i > 0$, $\F^{(i)} \E^{(i+1)} \F(\l+1) \cong \bigoplus_{j \ge 1} \F^{(j)} \E^{(j)} \otimes V_j'$ for some vector spaces $V_j'$. By Lemma \ref{th:notdirectsummand} every map from $ \id $ to  $\F^{(j)} \E^{(j)}$ is infinitesimal. So to check that $[ x_s^{\l}]$ is zero it suffices to check that $[ X^{\l}]$ is zero acting on the right factor in $\E \F(\l+1) \otimes V_0$. But 
$$\E\F(\l+1) \cong \id \otimes H^\star(\p^{\l-1}) \oplus \F\E(\l-1)$$ 
so $[ X^{\l}] = 0$ just by considering degrees.

Now suppose $i=s-1$. We know that $[ x_s^{\l}] = 0$ so that $[ x_s^{\l} t_{s-1}] = 0$. Using the nil affine Hecke relations we have 
\begin{eqnarray*}
x_s^{\l}t_{s-1}
&=& x_s^{\l-1}(t_{s-1} x_{s-1} - \id) \\
&=& x_s^{\l-2}(t_{s-1} x_{s-1} - \id) x_{s-1} - x_s^{\l-1} \\
&=& \dots \\
&=& t_{s-1} x_{s-1}^{\l} - (x_{s-1}^{\l-1} + x_{s-1}^{\l-2} x_s + \dots + x_s^{\l-1}).
\end{eqnarray*}
Since $t_{w_0}t_{s-1} = 0$ we get that 
\begin{equation}\label{eq:rec1}
[ t_{w_0}x_{s-1}^{\l-1}] = [ -t_{w_0}(x_{s-1}^{\l-2}x_s + \dots x_s^{\l-1})] \in [ t_{w_0} W_2^{s-2}].
\end{equation}

Now we multiply equation (\ref{eq:rec1}) on the right by $t_{s-2}$. Using the nil affine Hecke relations the left side becomes 
\begin{equation}\label{eq:rec2}
-t_{w_0}(x_{s-1}^{\l-2}+x_{s-1}^{\l-3}x_{s-2} + \dots + x_{s-2}^{\l-2}).
\end{equation}
The right side, since $x_s$ commutes with $t_{s-2}$, is a sum of terms of the form $t_{w_0}x_{s-2}^{a_{s-2}} x_{s-1}^{a_{s-1}} x_{s}^{a_s}$ where $a_i < \l-s+i$. Thus, isolating $[ t_{w_0}x_{s-2}^{\l-2}]$ from equation (\ref{eq:rec2}) we find that it must also be a sum of such terms and hence lies in $[ t_{w_0} W_2^{s-2}]$. 

Now we repeat in this way to get that $[ t_{w_0} x_i^{\l-s+i}] \in [ t_{w_0}W_2]$ for all $i$. At each step we make use of the fact that if $\alpha = t_{w_0}x_j^{a_j} \dots x_s^{a_s}$ where $a_k < \l-s+k$ then $\alpha t_{j-1}$ is a sum of terms of the form $t_{w_0}x_{j-1}^{a'_j} \dots x_s^{a'_s}$ where $a_k' < \l-s+k$. 
\end{proof}

\begin{Corollary} \label{lem:linearcomb}
Suppose $\l-2s \ge 0$. Let $W_2 \subset \Ext(\E^s \F^s(\l-s), \E^s \F^s(\l-s)) $ denote the subspace spanned by monomials $x_1^{a_1} \cdots x_s^{a_s}$ (acting on the $\F^s$ term) where $a_i < \l-s+i$. Then,
$$[ \pi \k[ x_1, \dots, x_s]] = [ \pi W_2]  \subset \Ext(\E^s \F^s, \E^s \F^{(s)}(\l-s))/\Inf.$$
\end{Corollary}
\begin{proof}
We repeatedly apply the above Lemma \ref{th:linearcomblem}.
 \end{proof}

\subsubsection{Case $r > 0$}

Now that we have done the $ r = 0 $ case, we will bootstrap up to $ r\ne 0 $.

\begin{Proposition} \label{cor:secondiso1}
Suppose $\l \ge 2r+s$. Then following map is an isomorphism 
\begin{equation*}
\F^{(r)}(\l) \otimes \pi W_2 \eta  \oplus \U \rightarrow \E^s \F^{(s+r)}(\l-s)
\end{equation*}
where $ W_2 \subset \End(\E^s \F^s \F^{(r)})$ is the subspace spanned by monomials $ x_1^{a_1} \cdots x_s^{a_s} $ acting on the $\F^s$ factor where $ a_i < \l-s+i$. More precisely, the map $\pi W_2 \eta$ is the composition
\begin{equation*}
\F^{(r)} \xrightarrow{\eta I} \E^s \F^s \F^{(r)} \xrightarrow{I W_2 I} \E^s \F^s \F^{(r)} \xrightarrow{I \pi} \E^s \F^{(s+r)}.
\end{equation*}
\end{Proposition}

\begin{proof}
We know that 
$$\E^s \F^{(s+r)}(\l + s) \cong \bigoplus_{i \ge 0} \F^{(i+r)}(\l-i) \E^{(i)} \otimes V_i$$
for some vector spaces $V_i$. By Lemma \ref{th:abstractiso},  we have that 
$$\dim(V_0) = s! \binom{\l}{s} = \frac{(\l)!}{(\l+s)!} \ge  \dim(W_2).$$

So if $\F^{(r)}(\l) \otimes \pi W_2 \eta \oplus \U \rightarrow \E^s \F^{(s+r)}(\l+s)$ is not an isomorphism then
there must be some projection map 
$$p: \E^s \F^{(s+r)}(\l+s) \rightarrow \F^{(r)}(\l)\langle k\rangle$$
onto some summand (for some $k \in \Z$) so that the composition 
$$\F^{(r)}(\l) \xrightarrow{\eta} \E^s \F^s \F^{(r)}(\l) \xrightarrow{\alpha} \E^s \F^s \F^{(r)}(\l) \xrightarrow{\pi} \E^s \F^{(s+r)} \xrightarrow{p} \F^{(r)}(\l)\langle k\rangle$$
is infinitesimal for any $\alpha \in W_2$. To simplify notation we will omit the $k$ as our convention dictates. 

But then, since $\F^r(\l) \cong \F^{(r)}(\l) \otimes H^\star({\mathbb{F}}l_r)$, the composition 
$$\F^r(\l) \xrightarrow{\eta} \E^s \F^s \F^r(\l) \xrightarrow{\alpha} \E^s \F^s \F^r(\l) \xrightarrow{\pi} \E^s \F^{(s+r)} \xrightarrow{p} \F^{(r)}(\l)$$
must also be infinitesimal for any $\alpha \in W_2$. Now let's precompose with a map $\F^r(\l) \xrightarrow{\beta} \F^r(\l)$ where $\beta = x_{s+1}^{a_{s+1}} \dots x_{s+r}^{a_{s+r}}$ and $a_i < \l-s+i$. Since this map commutes with $\eta$, we have 
$$\F^r(\l) \xrightarrow{\eta} \E^s \F^s \F^r(\l) \xrightarrow{\beta \circ \alpha} \E^s \F^s \F^r(\l) \xrightarrow{\pi} \E^s \F^{(s+r)} \xrightarrow{p} \F^{(r)}(\l)$$
which is also infinitesimal for any $\alpha \in W_2$ and $\beta$ as above. Notice that the fact this composition is $\F^{(r)}$-infinitesimal means that it is a combination of maps $\F^{(r)} \rightarrow \F^{(r)}\la k \ra$ where $k > 0$. 

Finally, let's precompose all the terms in this composition sequence with $\E^r$ to get 
$$\E^r \F^r(\l) \xrightarrow{\eta} \E^r \E^s \F^s \F^r(\l) \xrightarrow{\beta \circ \alpha} \E^r \E^s \F^s \F^r(\l) \xrightarrow{\pi} \E^r \E^s \F^{(s+r)} \xrightarrow{p} \E^r \F^{(r)}(\l).$$
Now both the left-most and right-most terms are isomorphic to a direct sum of terms $\E^{(r)} \F^{(r)}$ and the composition is made up of maps $\E^{(r)}\F^{(r)} \rightarrow \E^{(r)} \F^{(r)} \la k \ra$ with $k > 0$. This is because, as noted in the previous paragraph, the original maps were of the form $\F^{(r)} \rightarrow \F^{(r)} \la k \ra$ with $k > 0$. Now let
$$p': \E^{(r)}\F^{(r)} \rightarrow \id$$
be a projection map onto the $\id$ summand of highest graded degree. Then for $k > 0$ any map $\E^{(r)} \F^{(r)} \la k \ra \xrightarrow{p'} \id$ is $\id$-infinitesimal. Thus the composition 
$$\E^{(r)} \F^{(r)} \rightarrow \E^{(r)} \F^{(r)} \la k \ra \xrightarrow{p'} \id$$
must be $\id$-infinitesimal. 

We conclude that the composition 
$$\id \xrightarrow{\eta} \E^r \F^r(\l) \xrightarrow{\eta} \E^r \E^s \F^s \F^r(\l) \xrightarrow{\beta \circ \alpha} \E^r \E^s \F^s \F^r(\l) \xrightarrow{\pi} \E^r \E^s \F^{(s+r)} \xrightarrow{p} \E^r \F^{(r)}(\l) \xrightarrow{p'} \id$$
must be infinitesimal for any $\alpha \in W_2$ and $\beta$ as above. But now we can apply Lemma \ref{lem:linearcomb} which says that the space $\Ext(\id, \E^r \F^{(r)}(\l))/\Inf$ is spanned by maps of the form $\pi \gamma \eta$ where $\gamma = x_1^{a_1} \cdots x_{s+r}^{a_{s+r}}$ with $a_i < -(\l-r) - (s+r-i) = \l+s+i$. Since $\beta \circ \alpha$ is an example of such a $\gamma$, the map above could not have been infinitesimal, which is a contradiction.
\end{proof}

Note that a priori it is not clear that the map $ \k[x_1, \dots, x_s] \rightarrow \Ext(\F^s, \F^s) $ is injective.  However, consider the subspace $ W_2' \subset \k[x_1,\dots, x_s] $ spanned by 
$ x_1^{a_1} \dots x_s^{a_s} $ with $ a_i < \l - s + i$.  
We have a map $W_2 \rightarrow \Ext(\F^{(r)}(\l), \E^s \F^{(s+r)}(\l-s))/\Inf$ which takes $f \in W_2$ to the composition
\begin{equation*}
\F^{(r)} \xrightarrow{\eta I} \E^s \F^s \F^{(r)} \xrightarrow{I f I} \E^s \F^s \F^{(r)} \xrightarrow{I \pi} \E^s \F^{(s+r)}.
\end{equation*}
We denote this map by $f\mapsto  [\pi f \eta]$.
Then the preceding proposition immediately implies the following corollary.

\begin{Corollary} \label{th:injf}
The map $W_2 \rightarrow \Ext(\F^{(r)}(\l), \E^s \F^{(s+r)}(\l-s))/\Inf$ given by $f \mapsto [\pi f \eta]$ is injective.
\end{Corollary}

\begin{Proposition} \label{th:secondiso}
Suppose $\l-2s-r \ge 0$. Then there is an isomorphism 
\begin{equation*}
\F^{(r)}(\l) \otimes V  \oplus \U \rightarrow \E^{(s)}\F^{(s+r)}(\l + s)
\end{equation*}
where $ V \subset \Ext(\F^{(r)}, \E^{(s)} \F^{(s+r)})$ is a linear subspace such that the following diagram commutes
\begin{equation*}
\begin{CD}
\F^{(r)} @>>V> \E^{(s)}\F^{(s+r)} @>>\iota> \E^s\F^{(s+r)} \\
@| @. @AA\pi A \\
\F^{(r)} @>\eta>> \E^s \F^s \F^{(r)} @>W>> \E^s\F^s\F^{(r)}.
\end{CD}
\end{equation*}
Here $ W \subset \Ext(\F^s, \F^s) $ is the linear subspace spanned by those symmetric functions in $ x_1, \dots, x_s $ with no power of $ x_i $ greater than $ \lambda+s$, and $\U$ is the  direct sum of the ``other'' terms in the RHS of Proposition \ref{th:abstractiso} (recall it is a direct sum of $ \F^{(r+s)}\E^{(a)}$ with $ a > 0 $).

In particular, for each element $ f \in W$,  the composition $ \pi f \eta $ factors as $ \iota g $ for some $ g \in V$.  

\end{Proposition}
\begin{proof}
Suppose $f \in \Ext(\F^s \F^{(r)}, \F^{s}\F^{(r)})$ is a symmetric function in $x_1, \dots, x_s$. By Lemma \ref{th:zeroTi}, the composition 
$$\F^s \F^{(r)} \xrightarrow{f} \F^s \F^{(r)} \xrightarrow{\pi} \F^{(s+r)}$$
factors as 
$$\F^s \F^{(r)} \xrightarrow{\pi} \F^{(s)} \F^{(r)} \xrightarrow{h} \F^{(s)} \F^{(r)} \xrightarrow{\pi} \F^{(s+r)}$$
for some $g$. By adjunction (or equivalently, by applying $\E^s$ and precomposing with $\eta: \id \rightarrow \E^s \F^s$) this means that the map
$$\F^{(r)} \xrightarrow{\eta} \E^s \F^s \F^{(r)} \xrightarrow{f} \E^s \F^s \F^{(r)} \xrightarrow{\pi} \E^s \F^{(s+r)}$$
factors as 
$$\F^{(r)} \xrightarrow{\eta} \E^s \F^s \F^{(r)} \xrightarrow{\pi} \E^s \F^{(s)} \F^{(r)} \xrightarrow{h} \E^s \F^{(s)} \F^{(r)} \xrightarrow{\pi} \E^s \F^{(s+r)}.$$
Now, by Lemma \ref{lem:pieta} this is equal to 
$$\F^{(r)} \xrightarrow{\eta} \E^{(s)} \F^{(s)} \F^{(r)} \xrightarrow{\iota} \E^s \F^{(s)} \F^{(r)} \xrightarrow{h} \E^s \F^{(s)} \F^{(r)} \xrightarrow{\pi} \E^s \F^{(s+r)}.$$
Since $h$ acts only on the $\F^{(s)}$ term this means $h$ and $\iota$ commute so we get
$$\F^{(r)} \xrightarrow{\pi \circ h\circ \eta} \E^{(s)} \F^{(s+r)} \xrightarrow{\iota} \E^s \F^{(s+r)}.$$
Hence each element $ \pi f \eta $ factors as $ \iota g $ for $ g = \pi h \eta $.

It remains to show that the vector space $V$ of such maps induces an isomorphism (modulo $\U$). First notice that by Corollary \ref{th:injf} the only $\F^{(r)}$-infinitesimal map in $\pi W \eta$ is zero. Hence the only $\F^{(r)}$-infinitesimal map in $V$ is zero. It remains to show that we have enough maps in $V$ to give us an isomorphism. To do this we evaluate dimensions. 

On the one hand we have
$$\dim(V) = \dim(\pi W \eta) = \dim(W) = \binom{\l}{s}$$
since, from Corollary \ref{th:injf}, $\dim(\pi W_2 \eta) = \dim(W_2) $ and $W \subset W_2$. 

On the other hand, using the representation theory of $ \sl_2 $, we have
$$\dim_{\F^{(r)}}(\E^{(s)}\F^{(s+r)}(\l+s)) = \dim H^*(\bG(s,-(\l-r)-(s+r)+s)) = \dim H^*(\bG(s,\l)) = \binom{\l}{s}.$$
Thus $V$ induces an isomorphism.
\end{proof}

\begin{Lemma} \label{th:zeroTi}
Let $ f \in \Ext(\F^s, \F^s)$ be a symmetric polynomial in $ x_1, \dots, x_s $.  Then there exists a commutative diagram
\begin{equation*}
\begin{CD}
\F^s @>>f> \F^s \\
@VV\pi V @VV\pi V \\
\F^{(s)} @>h>> \F^{(s)}.
\end{CD}
\end{equation*}
\end{Lemma}
\begin{proof}
By Proposition \ref{th:firstiso}, if we consider compositions of the form 
$$\F^{(s)} \xrightarrow{\iota} \F^s \xrightarrow{x_1^{a_1} \cdots x_s^{a_s}} \F^s$$
then we pick out every direct summand of $\F^{(s)}$ inside $\F^s$. In fact, if we restrict to $\sum_i a_i < s(s-1)/2$ then we pick out all the copies except the ``top'' one (i.e. the one in degree $s(s-1)/2$). Since $\pi$ is the projection from this top direct summand we see that it suffices to prove that 
$$\pi f x_1^{a_1} \cdots x_s^{a_s} \iota = 0$$
whenever $\sum_i a_i < s(s-1)/2$. 

This is equivalent to showing that 
$$t_{w_0} f x_1^{a_1} \cdots x_s^{a_s} t_{w_0} = 0$$
where $w_0 \in S_s$ is the longest word which has length $s(s-1)/2$. On the other hand, if we slide the $x$s past the right hand $t_{w_0}$ using nil affine Hecke relation (iii) then we find that 
$$ x_1^{a_1} \cdots x_s^{a_s} t_{w_0} = \sum t_w x_1^{e_1} \cdot x_s^{e_s}$$
over some finite sum where $w \in S_s$. Since $\sum_i a_i < s(s-1)/2$ we conclude that $t_w \ne 1$. Thus it suffices to show that 
$$t_{w_0} f t_i = 0$$
for all $i=1, \dots, s$. 

Now if $ f $ is symmetric in the two variables $ x_i, x_{i+1} $, then $ ft_i = t_if $ by Lemma \ref{th:sym2}. So if $f$ is symmetric in all variables then $f t_i = t_i f$.  The result now follows since $t_{w_0} t_i = 0$ for all $i$. 
\end{proof}

\subsection{The terms in the complex}

Now we are in a position to examine carefully the terms in the complex $ \Theta_* \F^{(p)}(\lambda +p)$.  The basic ideas of this section were suggested to us by Joe Chuang.

Fix $ p $ such that $ \lambda -2p \ge N $. 

For each $ s = 0, \dots, p$, let us define a map 
\begin{equation*} 
\phi_s : \k[ x_1, \dots, x_p] \rightarrow \Ext(\F^{(\lambda+p)}(p), \F^{(\lambda+s)}(s)\E^s(\l + s)\F^{(p)}(\lambda +p) )
\end{equation*}
as the compostion $ \phi_s(f) = \pi \circ f \circ \eta  \circ \iota $ where we regard $ f $ as an element of $ \End(\F^{(\l+s)}\E^s \F^p) $,
\begin{equation*}
\F^{(\lambda + p)} \xrightarrow{\iota} \F^{(\lambda + s)} \F^{p-s} \xrightarrow{\eta} \F^{(\lambda +s)} \E^s \F^s \F^{p-s} \xrightarrow{f} \F^{(\lambda +s)} \E^s \F^s \F^{p-s} \xrightarrow{\pi} \F^{(\lambda +s)} \E^s \F^{(p)}
\end{equation*} 
If $ f= f_1 f_2 $, where $ f_1 \in \k[ x_1, \dots, x_s] $ and $ f_2 \in \k[x_{s+1}, \dots, x_p]$, then we can rewrite $ \phi_s(f) $ as a longer composition
\begin{equation} \label{eq:longcomp}
\begin{gathered}
\F^{(\l+p)} \xrightarrow{\iota} \F^{(\l+s)} \F^{p - s} \xrightarrow{f_2} \F^{(\l+s)} \F^{p - s} \xrightarrow{\pi} \F^{(\l+s)} \F^{(p - s)} \\ 
\xrightarrow{\eta} \F^{(\l+s)} \E^s \F^s \F^{(p- s)} \xrightarrow{f_1} \F^{(\l+s)} \E^s \F^s \F^{(p - s)} \xrightarrow{\pi} \F^{(\l+s)} \E^s \F^{(p)} 
\end{gathered}
\end{equation}
We will write this composition as $ \phi_s(f) = \phi_s^2(f_2) \phi_s^1(f_1) $, where $ \phi_s^1, \phi_s^2 $ denote the first and second lines of (\ref{eq:longcomp}) respectively.

Now, for each $ s$, let $ W_s $ be the subspace of $ \k[ x_1, \dots, x_p] $ which is symmetric functions in $ x_1, \dots, x_s $, each power of $x_i $ of degree less than or equal $ \lambda+p $, tensor with the span of monomials $ x_{s+1}^{t_{s+1}} \cdots x_p^{t_p} $ with $ 0 \le t_{s+1} < \dots < t_p < \lambda + p $.

\begin{Theorem} \label{th:vertiso}
There is an isomorphism $ \F^{(\lambda+p)} \otimes V_s \oplus \U \rightarrow \F^{(\lambda+s)}\E^{(s)}\F^{(p)}\langle- s\rangle=\Theta_s\F^{(p)} $  such that the diagram commutes
\begin{equation*}
\begin{CD}
\F^{(\lambda +p)}  @>V_s>>  \F^{(\lambda+s)}\E^{(s)}\F^{(p)}\langle -s\rangle \\
@|  @VV\iota V \\
\F^{(\lambda +p)} @>>\phi_s(W_s)> \F^{(\lambda+s)}\E^s\F^{(p)} \\
\end{CD}
\end{equation*}
Moreover, if $ f \in W_s $ and $ g \in V_s $ correspond to each other (i.e. $ \phi_s(f) = \iota g $), then $ \deg(g) = \deg(f) - p(\l+p-1) $. 
\end{Theorem}

\begin{proof}
This follows by combining Propositions \ref{th:firstiso} and \ref{th:secondiso} and the composition (\ref{eq:longcomp}).  Proposition \ref{th:secondiso} gives the decomposition of $ \E^{(s)}\F^{(p)} $ into $ \F^{(p - s)} $.  Then Proposition \ref{th:firstiso} gives the decomposition of $ \F^{(\lambda+s)} \F^{(p - s)} $ into $ \F^{(\lambda+p)}$. The composition (\ref{eq:longcomp}) shows that the two decompositions can be melded into the map $ \phi_s $.

For the computation of degrees, note that  
\begin{eqnarray*}
\deg(\phi_s(f)) 
&=& \deg(\pi) + \deg(f) + \deg(\eta) + \deg(\iota) \\
&=& -p(p-1)/2 + \deg(f) - s(\l + s) -(p-s)(2\l + p + s -1)/2 \\
&=& \deg(f) -p(\l+p-1) - s(s+1)/2.
\end{eqnarray*}
This also equals $\deg(g) + \deg(\iota \la -s \ra) = \deg(g) - s(s-1)/2 - s$ from which we get $\deg(g) = \deg(f) - p(\l+p-1)$.
\end{proof}

\subsection{The differential in the complex}

\begin{Proposition} \label{th:diff}
Let $ f \in W_s $.  Then the following diagram commutes
\begin{equation*}
\begin{CD}
\Theta_s \F^{(p)} @>>d_s> \Theta_{s-1} \F^{(p)}\\
@VV\iota V @VV\iota V \\
\F^{(\lambda+s)}\E^s \F^{(p)}\langle -s\rangle @>>\eta \circ \varepsilon>  \F^{(\lambda+s-1)}\E^{s-1} \F^{(p)}\langle -(s-1)\rangle \\
@AA\phi_s(f)A @AA\phi_{s-1}(f)A \\
\F^{(\lambda+p)} @= \F^{(\lambda+p)} 
\end{CD}
\end{equation*}
\end{Proposition}

\begin{proof}
As usual, to simplify notation we will drop all shifts. The top square above obviously commutes. So we must prove the commutativity of the bottom square.

The key point is that because $ (\eta, \varepsilon) $ are the unit and co-unit of adjunction the composition $\F \xrightarrow{\eta} \F \E \F \xrightarrow{\varepsilon} \F$ is the identity. In the large diagram below, all the squares obviously commute except for the larger rectangle on the right. Going around one side of the rectangle we have the composition of maps
$$\F \xrightarrow{\eta} \F \E \F \xrightarrow{\eta} \F \E (\E^{s-1} \F^{s-1}) \F \xrightarrow{f} (\F \E) \E^{s-1} \F^{s-1} \F \xrightarrow{\varepsilon} \E^{s-1} \F^{s-1} \F$$
which we can simplify since we can commute the map $\varepsilon$ from the right all the way to the left.  This gives
$$\F \xrightarrow{\eta} \F \E \F \xrightarrow{\varepsilon} \F \xrightarrow{\eta} \E^{s-1} \F^{s-1} \F \xrightarrow{f} \E^{s-1} \F^{s-1} \F,$$
which is the same as the composition $\F \xrightarrow{\eta} \E^{s-1} \F^{s-1} \F \xrightarrow{f} \E^{s-1} \F^{s-1} \F$ coming from going around the other side of the rectangle. Thus the large diagram below commutes.

\begin{equation*}
\begin{CD}
\F^{(\lambda+s)}\E^s \F^{(p)} @>>\iota> \F^{(\lambda +s-1)} \F \E^s \F^{(p)} @>>\varepsilon>  \F^{(\lambda+s-1)}\E^{s-1} \F^{(p)} \\
@AA\pi A @AA\pi A @AA\pi A \\
\F^{(\lambda +s)} \E^s \F^s \F^{p+s} @>>\iota> \F^{(\lambda +s-1)}\F \E^s \F^s\F^{p - s} @>>\varepsilon> 
\F^{(\lambda +s-1)}\E^{s-1}\F^{s-1}\F^{p - s+1} \\
@AAfA @AAfA @AAfA \\
\F^{(\lambda +s)} \E^s \F^s \F^{p+s} @>>\iota> \F^{(\lambda +s-1)}\F \E^s \F^s\F^{p - s} @. 
\F^{(\lambda +s-1)}\E^{s-1}\F^{s-1}\F^{p - s+1} \\
@AA\eta A @AA\eta A @AA\eta A \\
\F^{(\l +s)} \F^{p - s} @>>\iota> \F^{(\l +s-1)} \F^{p - s+1} @= \F^{(\l +s-1)}\F^{p-s+1} \\
@AA\iota A @AA\iota A @AA\iota A \\
\F^{(\lambda+p)} @= \F^{(\lambda+p)} @= \F^{(\lambda+p)} 
\end{CD}
\end{equation*}

The commutativity of the whole diagram gives the commutativity of the bottom square in the statement of the proposition.
\end{proof}

We consider the modified Koszul complex following Chuang and Rouquier.  Fix a graded vector space $ M $ and an integer $ l \le \dim M$.  Then, there is a complex of graded vector spaces $ C_s := Sym^s M \otimes \Lambda^{l-s} M $, where $ s $ varies from $ 0, \dots, l$.  We have maps $ C_{s+1} \rightarrow C_s $, taking parts of the sym and sending them over to the wedge.  This complex is exact.  

Now, fix a vector $ v \in M$, and consider the subcomplex $ C'_s := Sym^s M \otimes\Lambda^{l-s-1} M \wedge v \subset Sym^s M \otimes \Lambda^{l-s} M = C_s $ with $ s = 0, \dots, l-1 $. This complex has homology in precisely one degree, namely $ s = l-1$, since $ v^{l-1} \otimes v $ is sent to zero.  

Let $ M = \k^{\l + p +1} $ denote a graded vector space of dimension $ \lambda + p + 1 $ with basis $ e_0, e_1, \dots e_{\lambda +p} $ where $ e_i $ has degree $ 2i $.  Let $ v = e_{\lambda + p}$.  Let $ l = p +1 $ and let $ C'_s $ be the complex as above with $ s $ varying from $ 0 $ to $ p $. 

\begin{Proposition} \label{th:ThetaC}
The $\F^{(\l+p)}$-part of the complex $\Theta_* \F^{(p)}(\l+p) $ modulo $\F^{(\l+p)}$ infinitesimal maps is isomorphic to the complex $ C'_* \la p(\l+p-1)+2(\l+p) \ra$ for the vector space $M$ as above.
\end{Proposition}

Before giving the proof, let us explain the intuition behind this statement.  From Theorem \ref{th:vertiso}, we have an isomorphism between the terms in the complex and direct sums of $ \F^{(\l + p)} $.  Each isomorphism is encoded by the map $ \phi_s $ and a space of polynomials $ W_s$.  Each space of polynomials $ W_s $ is the tensor product of a ``symmetric part'' in the variables $ x_1, \dots, x_s $ and an ``anti-symmetric'' part in the variables $ x_{s+1}, \dots, x_p $.  When we apply the differential, we use the same polynomials, but we use the map $ \phi_{s-1} $ instead of $ \phi_s $ (this is the content of Proposition \ref{th:diff}).  The effect of this is that the variable $ x_s $ is now regarded in the anti-symmetric part instead of the symmetric part.  This is analogous to the differential in the complex $ C' $.  Finally, the presence of the $ e_{\l + p} \wedge $ in the complex $ C'$ is due to the fact that in the symmetric part $ W_s $ we allow the variables to have degree less than or equal to $ \l + p $ while in the anti-symmetric part they have degree less than $ \l + p $. 

The proof of this Proposition is adapted from the proof of Theorem 6.6 in \cite{CR}.

\begin{proof}
In this proof infinitesimal always means $\F^{(\l+p)}$-infinitesimal. 

First we will define an isomorphism $ C'_s\la 2(\l + p)\ra \rightarrow W_s$ by
\begin{equation} \label{eq:symwedge}
\begin{aligned}
 Sym^s \k^{\lambda +p + 1} \otimes \Lambda^{p - s} \k^{\lambda+p +1} \wedge e_{\l + p} \rightarrow W_s \\
e_{a_1} \cdots e_{a_s} \otimes  e_{b_{s+1}} \wedge \cdots \wedge e_{b_p} \wedge e_{\l +p}\mapsto  m_{a_1 \cdots a_s} x_{s+1}^{b_{s+1}} \cdots x_p^{b_r}
\end{aligned}
\end{equation}
where $ m_{a_1 \cdots a_s} $ denotes the monomial symmetric function in the variables $ x_1, \dots, x_s $ (i.e. the symmetrization of $ x_1^{a_1} \cdots x_s^{a_s} $) and where $ a_1 \le \cdots \le a_s \le \l + p $ and $ b_{s+1} < \cdots < b_p < \l +p $.

From the statement of Theorem \ref{th:vertiso}, there is an isomorphism $ V_s \cong W_s\la p(\l + p-1) \ra $.  Combining this with the above isomorphism, we get an isomorphism $ V_s \cong C'_s \la 2(\l+p) + p(\l+p -1) \ra $.

So now we just have to calculate what happens to basis vectors.  Let us pick $f =  m_{a_1, \dots, a_s} x_{s+1}^{b_{s+1}} \cdots x_p^{b_p} \in W_s $.  Here $ m_{a_1, \dots, a_s} $ denotes a monomial symmetric function in $ x_1, \dots, x_s $.  This basis element corresponds to $ e_{a_1} \cdots e_{a_s} \otimes e_{b_{s+1}} \wedge \cdots \wedge e_{b_p} \wedge e_{\lambda + p} \in C'_s$.  

Let $ g $ denote the corresponding element of $ V_s $.  Then by Proposition \ref{th:diff}, $ \iota d_s g = \phi_{s_1}(f) $.

Now, 
\begin{equation*}
f = m_{a_1, \dots, a_s} x_{s+1}^{b_{s+1}} \cdots x_p^{b_p}  = \sum_{i=1}^s m_{a_1, \dots, \hat{a_i}, \dots, a_s} x_s^{a_i} x_{s+1}^{b_{s+1} }\cdots x_p^{b_p}
\end{equation*}

Hence $$ \phi_{s-1}(f) = \sum_{i=1}^s  \phi_{s-1}^2( x_s^{a_i} x_{s+1}^{b_{s+1}} \cdots x_p^{b_p})\phi_{s-1}^1(m_{a_1, \dots, \hat{a_i}, \dots, a_s}). $$
Comparing with the differential in the modified Koszul complex, we see that to prove the desired result, we must show that 
\begin{equation} \label{eq:mustshow}
 \phi_{s-1}^2( x_s^{a_i} x_{s+1}^{b_{s+1}} \cdots x_p^{b_p}) = \begin{cases}
 0 &\text{ if } a_i \in \{b_{s+1}, \dots, b_p, \l +p\} \\
 \pm \phi_{s-1}^2(x_s^{c_1} \cdots x_p^{c_s}) &\text{ otherwise, where } c_1 <  \dots < c_s \text{ and } \{c_1, \dots, c_s \} = \{a_i, b_{s+1}, \dots, b_p \} 
 \end{cases}
 \end{equation}
(The above equalities are considered modulo infinitesimals.)
 
To establish (\ref{eq:mustshow}), we first apply Lemma \ref{th:vanish} with $k=p-s+1$ to see that $ \phi_{s-1}^2(x_s^{a_i} x_{s+1}^{b_{s+1}} \cdots x_p^{b_p}) $ is infinitesimal if $ a_i = \l+p $. More precisely, $\phi_{s-1}^2(x_s^{p-s+1})$ is the composition 
$$\F^{(\l+p)} \xrightarrow{\iota} \F^{(\l+s-1)} \F^{p-s+1} \xrightarrow{I x_1^{p-s+1} I} \F^{(\l+s-1)} \F^{p-s+1} \xrightarrow{\pi} \F^{(\l+s-1)} \F^{(p-s+1)}$$
which is infinitesima by Lemma \ref{th:vanish}. Thus $\phi_{s-1}^2(x_s^{\l+p})$ is also infinitesimal since $\l+p \ge p-s+1$ (we have $s-1 \ge 0$).  

Then we apply Corollary \ref{th:sym2c} to see that $ \phi_{s_1}^2(\cdots x_j^a x_{j+1}^b \cdots) =  -\phi_{s_1}^2(\cdots x_j^b x_{j+1}^a \cdots)$. This allows us to see that the LHS of (\ref{eq:mustshow}) vanishes when $ a_i \in \{b_{s+1}, \dots, b_p \} $ and also allows us to reorder the exponents as in the second case of (\ref{eq:mustshow}). 
\end{proof}

\begin{Lemma} \label{th:vanish}
The map $ x_1^k \iota : \F^{(k)} \rightarrow \F^k\langle -k(k-1)\rangle \rightarrow \F^k\langle 2k-k(k-1)\rangle $ is $\F^{(k)}$-infinitesimal.
\end{Lemma}

\begin{proof}
We can factor the map as
\begin{equation*}
\F^{(k)} \xrightarrow{\iota} \F \F^{(k-1)}\langle -k+1\rangle \xrightarrow{x^k} \F \F^{(k-1)}\langle k+1\rangle \xrightarrow{\iota} \F^k\langle 2k-k(k-1)\rangle
\end{equation*}
Now $\F \F^{(k-1)} \cong \F^{(k)} \otimes_\k H^\star(\p^k)$ so $x^k \iota$ is infinitesimal by degree considerations. Subsequently the whole composition is also infinitesimal.
\end{proof}

\begin{proof}[Proof of \ref{th:Tonhighweight2}]
By Corollary \ref{th:abstractiso2}, the terms in the complex are direct sums of $ \F^{(\l+p)} $ and $ \F^{(a)} \E^{(b)} $ with $ b \ne 0 $.  By Proposition \ref{th:ThetaC}, the $\F^{(\l+p)}$-part of the complex (modulo infinitesimal maps) is isomorphic to $ C'_\bullet\la p(\l+p-1) +2(\l+p) \ra$. 

Now the homology of $ C'_\bullet $ is zero except in homological degree $p$ where it is one dimensional and is represented by $e_{\l+p}^p \otimes e_{\l+p}$.  This element has degree $2(p+1)(\l+p)$. 

Under the isomorphism from Proposition \ref{th:ThetaC}, $e_{\l+p} \dots e_{\l+p} \otimes e_{\l+p}$ corresponds to the following composition
$$\F^{(\l+p)} \la -p(\l+p+1) \ra \xrightarrow{\eta} \F^{(\l+p)} \E^p \F^p \la -2p(\l+p)-p \ra \xrightarrow{(x_1 \dots x_p)^{\l+p}} \F^{\l+p} \E^p \F^p \la -p \ra.$$
More precisely, by Theorem \ref{th:vertiso} we have a commutative diagram 
\begin{equation*}
\begin{CD}
\F^{(\l+p)} \la -p(\l+p+1) \ra @>>\eta> \F^{(\l+p)}\E^p\F^p  \la -2p(\l+p)-p \ra @>>(x_1 \dots x_p)^{\l+p}> \F^{(\l+p)}\E^p\F^p \la -p \ra \\
@| @.  @VV\pi V \\
\F^{(\l+p)} \la -p(\l+p+1) \ra  @>v_p \circ \eta>>  \F^{(\l+p)}\E^{(p)}\F^{(p)} \la -p \ra @>\iota>> \F^{(\l+p)}\E^p\F^{(p)} \la -p-\frac{p(p-1)}{2} \ra \\
\end{CD}
\end{equation*}
for some $v_p \in \End^{2p(\l+p)}(\F^{(\l+p)}\E^{(p)}\F^{(p)})$. Let $ \gamma_p = \eta \circ v_p$.

Now let $ A \in \D(\l+2p) $ be a highest weight object.  Then $ \E(A) = 0 $, so $ \F^{(a)}\E^{(b)}(A) = $ when $ b>0 $.  Thus for each $ s $ by Theorem \ref{th:vertiso}, we have isomorphisms 
\begin{equation*}
\Theta_s(\F^{(p)}(A)) \cong \F^{(\l+p)}(A) \otimes V_s
\end{equation*}

The maps in the complex $ \Theta_\bullet (\F^{(p)}(A)) $ are induced by the maps in the $ \F^{(\l+p)} $ part of the complex $\Theta_\bullet \F^{(p)}$.  Hence with respect to these fixed isomorphisms, there are no negative degree maps and also the 0-degree maps come from maps between the vector spaces $ V_s $.  Thus, we are in the situation of section \ref{sec:exactcomplex}.  

We now consider the extended complex
\begin{equation*}
\F^{(\l+p)}(A)\la -p(\l+p+1) \ra \xrightarrow{\gamma_p} \Theta_p \F^{(p)}(A) \rightarrow \dots \rightarrow \Theta_0 \F^{(p)}(A)
\end{equation*}
Again we have no negative degree maps and the degree 0 maps come from the complex in vector spaces
\begin{equation*}
\k \la -p(\l+p+1) \ra \rightarrow V_p \rightarrow \dots \rightarrow V_0
\end{equation*}

This complex is exact since by Theorem \ref{th:ThetaC}, it is isomorphic to the complex
\begin{equation*}
\k \la -p(\l+p+1) \ra \rightarrow C'_p \rightarrow \dots \rightarrow C'_0
\end{equation*}
which is exact.
\end{proof}

\section{Proof of Theorem \ref{th:equiv}} \label{se:mainproof}

In this section we move to the setting where all the categories are enhanced triangulated categories. In this setting the shift $\la k \ra$ is equal to $[k]\{-k\}$ where $[\cdot]$ is the natural homological shift. 
 
\subsection{Equivalences between triangulated categories}

Recall that a spanning class in a triangulated category $ \D $ is a subset $ S $ of the set of objects which is closed under shifts $[\cdot]$ and such that if $ A \in \D$ is an object and $ Hom(B, A) = 0 $ for all $ B \in S $, then $ A = 0 $.

The following result and its proof are adapted from Huybrechts \cite[Prop 1.49]{H}.

\begin{Lemma}  \label{th:spanequiv}
Let $ \Phi : \D \rightarrow \D' $ be an exact functor between triangulated categories. Suppose that the left and right adjoint of $\Phi$ are isomorphic: $\Phi_R \cong \Phi_L$. Suppose also that $S$ is a spanning class such that for all $B \in S$ the natural adjunction map $\varepsilon : \Phi_L \circ \Phi(B) \rightarrow B $ is an isomorphism.  Then $ \Phi $ is an equivalence.
\end{Lemma}

\begin{proof}
We first show that $\Phi$ is fully faithful. By Remark 1.24 from \cite{H} it is enough to show that $\eta: \id \rightarrow \Phi_R \circ \Phi$ is an isomorphism. By the Yoneda lemma it suffices to show that $ \eta_A : A \rightarrow \Phi_R \circ \Phi(A) $ is an isomorphism for all $A \in \D$.

Let $ Z(A) = \Cone(\eta_A) $.  To show that $\eta_A$ is an isomorphism, it suffices to show that $Z(A) = 0$. To do so, we would like to prove that $ \Hom(B, Z(A)) = 0 $ for all $ B \in S$. 

For $ B \in S$, we have a long exact sequence
\begin{equation*}
\cdots \rightarrow \Hom(B, A) \rightarrow \Hom(B, \Phi_R \circ \Phi(A)) \rightarrow \Hom(B, Z(A)) \rightarrow \Hom(B\langle -1\rangle, A) \rightarrow \cdots
\end{equation*}
Hence it suffices to show that $ \Hom(B, A) \rightarrow \Hom(B, \Phi_R \circ \Phi(A) ) $ is an isomorphism for all $ B \in S$.  However, by adjunction $ \Hom(B, \Phi_R \circ \Phi(A)) \cong \Hom(\Phi(B), \Phi(A)) $.  So it suffices to show that $ \Hom(B, A) \rightarrow \Hom(\Phi(B), \Phi(A)) $ is an isomorphism. 

Now applying adjunction again we see that it suffices to show that $ \Hom (B,A) \rightarrow \Hom(\Phi_L \circ \Phi(B), A) $ is an isomorphism.  But since $ \Phi_L \circ \Phi(B) \rightarrow B $ is an isomorphism by hypothesis, we are done.  This prove that $\Phi$ is fully faithful. 

Now, since $\Phi_L \cong \Phi_R$ this means that $\Phi_L(B) = 0$ implies $\Phi_R(B) = 0$ for any $B \in \D'$. Then by Proposition 1.54 of \cite{H} this implies $\Phi$ is an equivalence. 
\end{proof}

\subsection{Proof of the equivalence}

Fix $ \lambda \ge 0$. Recall that the complex $\Theta_s$ has terms 
$$\Theta_s(\l) = \F^{(\l+s)}(s)\E^{(s)}(\l+s) [-s] \{s\}$$ 
where $s = 0, \dots, (N-\l)/2$. 

\begin{Proposition}\label{prop:uniqueconv}
The complex $ \Theta_*$ has a unique convolution.
\end{Proposition}
\begin{proof}
We just need to check that both conditions in Proposition \ref{th:uniquecone} hold. To simplify the notation we will omit the second grading $\{ \cdot \}$ (notice that by omitting the second grading we are actually proving a stronger, rather than weaker, result).

Switching indices, the first condition of \ref{th:uniquecone} says that $\Hom(A_s[k-1], A_{s-k})=0$ for $k \ge 2$ and $s \ge 0$. This is equivalent to 
$$\Hom(\F^{(\l+s)}(s)\E^{(s)}(\l+s) [-s][k-1], \F^{(\l+s-k)}(s-k)\E^{(s-k)}(\l+s-k)[-s+k]) = 0.$$
By adjunction the left side is equal to 
$$\Hom(\E^{(s)}(\l+s), \F^{(\l+s)}(s)_R \F^{(\l+s-k)} \E^{(s-k)} [1])$$
which is the same as
$$\Hom(\E^{(s)}, \E^{(\l+s)}(s) [-(\l+s)s] \F^{(\l+s-k)}(s-k) \E^{(s-k)} [1]).$$
Now, by Lemma \ref{th:abstractiso} we have 
$$\E^{(\l+s)}(s) \F^{(\l+s-k)}(s-k) \cong \bigoplus_{j \ge 0} \F^{(a-j)}\E^{(b-j)} \otimes_\k H^\star(\bG(j, \l+2s-k))$$
where $a=\l+s-k$ and $b=\l+s$. So we just need to check that 
\begin{equation}\label{eq:4}
\Hom(\E^{(s)}(\l+s), \F^{(a-j)}(s-k+2b-j)\E^{(b-j)} \E^{(s-k)} \otimes_\k H^\star(\bG(j,\l+2s-k)) [-(\l+s)s+1]) = 0
\end{equation}
for $j=0, \dots, \l+s-k$. Now by adjunction the left side is equal to 
$$\Hom(\E^{(a-j)}[(a-j)(s-k+2b-j)] \E^{(s)}, \E^{(b-j)} \E^{(s-k)} \otimes_\k H^\star(\bG(j,\l+2s-k)) [-(\l+s)s+1]).$$
Finally we use that 
$$\E^{(a-j)} \E^{(s)} \cong \E^{(a-j+s)} \otimes_\k H^\star(\bG(s,a-j+s)) \text{ and } \E^{(b-j)} \E^{(s-k)} \cong \E^{(a-j+s)} \otimes_\k H^\star(\bG(s-k,a-j+s)).$$
This means that the copies of $\E^{(a-j+s)}$ inside 
$$\E^{(a-j)} [(a-j)(s-k+2b-j)] \E^{(s)}$$
occur in degrees 
$$d \le s(a-j) - (a-j)(s-k+2b-j) = (a-j)(k-2b+j)$$
while the copies of $\E^{(a-j+s)}$ inside 
$$\E^{(b-j)} \E^{(s-k)} \otimes_\k H^\star(\bG(j,\l+2s-k)) [-(\l+s)s+1]$$
occur in degrees
$$d' \ge -(b-j)(s-k) - j(s-k+b-j) + (\l+s)s - 1 = b(k - j) + j^2 - 1.$$
Here we used that $H^\star(\bG(m,n))$ is supported in degrees $-m(n-m) \le d \le m(n-m)$. 

So to obtain the vanishing in (\ref{eq:4}) it suffices to show that $d < d'$ which is the same as 
$$(a-j)(k-2b+j) < b(k-j) + j^2 -1$$
which simplifies to 
$$2j^2 + 2j(-2b+k) + k^2 + 2ab - 1 > 0.$$
The discriminant of this quadratic in $j$ is 
$$4(4b^2-4bk+k^2) - 8(k^2+2(b-k)b-1) = -4k^2+8$$
which is negative for $k \ge 2$ so we are done. 

The second condition is simila, resulting in an expression which simplifies to 
$$2j^2+2j(-2b+k)+k^2+2ab-2 > 0$$
where this time $k \ge 3$. The discriminant of the left hand side is then $-4k^2+16$, which is negative for $k \ge 3$. 
\end{proof}

We let $ \T : \D(\l) \rightarrow \D(-\l)$ denote the convolution of $\Theta_*$. 

As usual assume $\l \ge 0$. In order to show that $\T$ is an equivalence we use the following class of objects. Let $S$ denote the objects in $\D(\l) $ which are of the form $ \F^{(p)}(A) $ for some highest weight object $A \in \D(\l+2p) $, where $ p $ is allowed to range over $ 0, \dots, (N-\l)/2 $.

Recall that the value of $ \T $ on these objects is given by Theorem \ref{th:Tonhighweight}, which we will now prove.
\begin{proof}[Proof of Theorem\ref{th:Tonhighweight}]
Consider the extended complex
\begin{equation} \label{eq:extendcomp}
\F^{(\l+p)}(A)\la- p(\l + p +1) \ra \xrightarrow{\gamma_p} \Theta_p(\F^{(p)}(A)) \rightarrow \cdots \rightarrow \Theta_0(\F^{(p)}(A)
\end{equation}
By Theorem \ref{th:Tonhighweight2}, this complex is exact in the sense of section \ref{sec:exactcomplex}.  

Consider a Postnikov system $ B_i $ for the complex $ \Theta_\bullet $.  Applying these $ B_i $ to $ F^{(p)}(A) $ gives us a partial Postnikov system for the extended complex (\ref{eq:extendcomp}) (here we are using our axioms of an enhanced strong categorical $ \sl_2 $ action).  By Lemma \ref{lem:conezero}, this partial Postnikov system can be extended to a Postnikov system for (\ref{eq:extendcomp}) which give 0 as the convolution.

In particular, this means that there is a distinguished triangle
\begin{equation*}
0 \rightarrow \F^{(\l+p)}(A)\la -p(\l+p+1) \xrightarrow{\gamma'_p} \T(\F^{(p)}(A))[-p]
\end{equation*}
where $ \gamma'_p $ factors the map $ \gamma_p $ as $ \F^{(\l+p)}(A)\la -p(\l+p+1) \xrightarrow{\gamma'_p} \T(\F^{(p)}(A))[-p] \rightarrow \Theta_p \F^{(p)}(A) $.

This implies that $ \T(\F^{(p)}(A)) \cong \F^{(\l+p)}(A)[-p(\l+p)]\{p(\l+p+1)\} $ (recall that $ \la 1 \ra = [1]\{-1 \} $).
\end{proof}

\begin{Lemma}
$S$ forms a spanning class for $\D(\l)$.
\end{Lemma}
\begin{proof}
We proceed by a decreasing induction on $\l$. The result is obvious when $\l = N $.  So now assume it holds for $ \l+2$. We need to show it also holds for $\l$.  Let $ A \in \D(\l) $ and assume that $ \Hom(B, A) = 0 $ for all $ B \in S $.  For highest weight object  $ C \in  \D(\l + 2p) $ we have 
$$ \F \F^{(p-1)}(C) \cong \oplus_{i=0}^{p-1} \F^{(p)}(C) [p-1-2i] $$ 
and so $ \Hom(\F \F^{(p-1)}(C), A) = 0 $. By adjunction, this implies that $ \Hom(\F^{(p-1)}(C), \E(A))=0 $.  Now $ \E(A) \in \D(\l+2) $, so by induction, we conclude that $\E(A) = 0$. Hence $A \in S$ (since $A$ is of highest weight), which means $\Hom(A,A) = 0$, implying that $A = 0$ (as desired).
\end{proof}

The following proposition completes the proof the Theorem \ref{th:equiv}.
\begin{Proposition}
$\T $ is an equivalence of categories.
\end{Proposition}
\begin{proof}
We are going to use Lemma \ref{th:spanequiv}. First we need to check that $\T_R \cong \T_L$. Recall that $\T$ is the convolution of $\Theta_*$ where $\Theta_s = \F^{(\l+s)}(s) \E^{(s)}(\l+s) [-s]\{s\}$ and $s=0, \dots, (N-\l)/2$. So 
\begin{eqnarray*}
{\Theta_s}_L 
&=& \E^{(s)}(\l+s)_L \F^{(\l+s)}(s)_L [s]\{-s\} \\
&\cong& \F^{(s)}(\l+s) \la -s(\l+s) \ra \E^{(\l+s)}(s) \la (\l+s)s \ra [s]\{-s\} \\
&\cong& \F^{(s)}(\l+s) \E^{(\l+s)}(s) [s]\{-s\}.
\end{eqnarray*}
Similarly, 
\begin{eqnarray*}
{\Theta_s}_R 
&=& \E^{(s)}(\l+s)_R \F^{(\l+s)}(s)_R [s]\{-s\} \\
&\cong& \F^{(s)}(\l+s) \la s(\l+s) \ra \E^{(\l+s)}(s) \la -(\l+s)s \ra [s]\{-s\} \\
&\cong& \F^{(s)}(\l+s) \E^{(\l+s)}(s) [s]\{-s\}.
\end{eqnarray*}
Hence the terms in the complexes for $\T_R$ and $\T_L$ are isomorphic. To see that the connecting maps are the same we (up to a non-zero multiple) we go back to their definition and note that they are made up of compositions involving $\iota$ and $\varepsilon$. Then we just use the fact that the right and left adjoints of the maps $\iota$ and $\varepsilon$ are the same (up to a non-zero multiple). For example, $\iota_R$ and $\iota_L$ are the same because they both represent maps which are unique (up to a non-zero multiple) and similarly with $\varepsilon$. Thus $\T_R \cong \T_L$. 

The second thing is we need to check that $\T_L \circ \T(A) \rightarrow A$ is an isomorphism for any $A \in S$. So let $ A \in \D(\l+2p)$ be a highest weight object. Then $ \T(\F^{(p)}(A)) \cong \F^{(\l+p)}(A)[p(\l+p)]\{-p(\l+p+1)\}$ by Theorem \ref{th:Tonhighweight}. Note that this isomorphism is induced by the map 
$$\gamma_p: \F^{(\l+p)} [-p(\l+p+1)]\{p(\l+p+1)\} \rightarrow \Theta_p \F^{(p)}.$$ 
as described in Theorem \ref{th:Tonhighweight2}.

Now, by Lemma \ref{th:abstractiso}, 
$$\E^{(p)} \F^{(\l+2p)}(0) : \D(\l+2p) \rightarrow \D(-\l)$$ 
is equal to $\F^{(\l+p)}(p) \oplus \U$ where $\U$ is a direct sum of terms $\F^{(b)} \E^{(a)}$ ($a > 0$). So 
$$\F^{(\l+p)}(A) [-p(\l+p+1)]\{p(\l+p+1)\} \cong \E^{(p)} \F^{(\l+2p)}(A) [-p(\l+p+1)]\{p(\l+p+1)\}$$
where $\F^{(\l+2p)}(A) \in \D(-\l-2p)$ is a lowest weight object, meaning $\F \F^{(\l+2p)}(A)=0$ (see Lemma \ref{lem:low-high} below). We will denote by $\beta$ the map 
$$\E^{(p)} \F^{(\l+2p)}(0) \xrightarrow{\beta} \F^{(\l+p)}.$$ 
which induces the isomorphism above when applied to highest weight objects $A$. 

Now the same analysis used in Theorem \ref{th:Tonhighweight} can be used to show that if $B \in \D(-\l-2p)$ is a lowest weight object then $\T_L \E^{(p)} (B) \cong \E^{(\l+p)}(B) [p(\l+p)]\{-p(\l+p+1)\}$ and is induced by some map 
$$\E^{(\l+p)} [p(\l+p+1)]\{-p(\l+p+1)\} \xrightarrow{\gamma_p'} {\Theta_p}_L \E^{(p)}.$$

If we combine all of this then we get 
\begin{eqnarray*}
\T_L \T \F^{(p)}(A) 
&\cong& \T_L \F^{(\l+p)}(A) [-p(\l+p)]\{p(\l+p+1)\} \\
&\cong& \T_L \E^{(p)} \F^{(\l+2p)}(A) [-p(\l+p)]\{p(\l+p+1)\} \\
&\cong& \E^{(\l+p)} \F^{(\l+2p)}(A) \\
&\cong& \F^{(p)}(A).
\end{eqnarray*}
To get the last isomorphism we use Lemma \ref{th:abstractiso} to conclude that $\E^{(\l+p)} \F^{(\l+2p)} \cong \F^{(p)} \oplus \U$ where $\U$ is of the form $\F^b \E^a$ ($a > 0$). As before, we also pick a map $\F^{(p)} \xrightarrow{\alpha} \E^{(\l+p)} \F^{(\l+2p)}$ which induces this last isomorphism. 

It remains to show that the adjunction map $\T_L \T \F^{(p)}(A) \rightarrow \F^{(p)}(A)$ induces this isomorphism. To do this we note, by reading the sequence of isomorphisms backwards, that the isomorphism $\F^{(p)}(A) \rightarrow \T_L \T \F^{(p)}(A)$ is induced by the composition 
\begin{eqnarray*}
\F^{(p)} \xrightarrow{\alpha} \E^{(\l+p)} \F^{(\l+2p)} &\xrightarrow{\gamma_p'}& \T_L \E^{(p)} \F^{(\l+2p)} [-p(\l+p)]\{p(\l+p+1)\} \\
&\xrightarrow{\beta}& \T_L \F^{(\l+p)} [-p(\l+p)]\{p(\l+p+1)\} \xrightarrow{\gamma_p} \T_L \T \F^{(p)}.
\end{eqnarray*} Now the composition $\F^{(p)} \xrightarrow{\gamma_p \beta \gamma_p' \alpha} \T_L \T \F^{(p)} \xrightarrow{\mbox{adj}} \F^{(p)}$ is either zero or (a non-zero multiple of) the identity because $\End(\F^{(p)}) \cong \k \cdot \id$. Now $\gamma_p \beta \gamma_p' \alpha: \F^{(p)}(A) \rightarrow \T_L \T \F^{(p)}(A)$ is an isomorphism so if the composition is zero then the induced map $\T_L \T \F^{(p)}(A) \xrightarrow{\mbox{adj}} \F^{(p)}(A)$  must be zero, which is absurd since it is the counit of an adjunction. Thus the composition must be an isomorphism and hence $\T_L \T \F^{(p)}(A) \xrightarrow{\mbox{adj}} \F^{(p)}(A)$ is also an isomorphism.
\end{proof}

\begin{Lemma}\label{lem:low-high} Let $\mu \ge 0$ and suppose $A \in \D(\mu)$ is a highest weight object. Then $\F^{(\mu)}(A) \in \D(-\mu)$ is a lowest weight object. Similarly, if $B \in \D(-\mu)$ is a lowest weight object then $\E^{(\mu)}(B) \in \D(\mu)$ is a highest weight object. 
\end{Lemma}
\begin{proof}
For the first claim we need to show that $\F \F^{(\mu)}(A) = 0$. Since $\F \F^{(\mu)} = \F^{(\mu+1)} \otimes_k H^\star(\p^{\mu})$ it suffices to show $\F^{(\mu+1)}(A)=0$. Now let's consider 
$$\F^{(\mu+2)}\E: \D(\mu+1) \rightarrow \D(-\mu-1).$$ 
By Lemma \ref{th:abstractiso} this is equal to $\F^{(\mu+1)} \oplus \U$ for some $\U$. But $\E(A)=0$ since $A$ is of highest weight so we get $\F^{(\mu+1)}(A) \oplus \U(A) = 0$. This means $\F^{(\mu+1)}(A) = 0$ and we are done. 

The assertion involving $B$ is proven in the same way. 
\end{proof}

\section{Equivalences $D(T^\star \bG(k,N)) \xrightarrow{\sim} D(T^\star \bG(N-k,N))$}\label{sec:application}

In \cite{ckl2} we described a strong categorical $\sl_2$ action on derived categories of cotangent bundles to Grassmannians. We briefly review that construction and then discuss the induced derived equivalence 
$$\T: D(T^\star \bG(k,N)) \xrightarrow{\sim} D(T^\star \bG(N-k,N))$$
via the construction of $\Theta$ given above. 

If the functors $\E, \F$ or $\T$ are FM tranforms then the corresponding kernels will be written using calligraphic font as $\sE, \sF$ and $\sT$. 

\subsection{Varieties and functors}

Fix $ N > 0 $. For each $ \l = -N, -N+2, \dots, N$, we set $ Y(\l) = T^\star \bG(k,N)$, the total cotangent bundle to the Grassmannian $T^\star \bG(k,N)$ where  $k=(N-\l)/2$.  There is an action of $ \C^\times $ on these varieties, scaling the fibres of these cotangent bundles.  We define our categories $ \D(\l) = D^{\C^\times} Coh(Y(\l)) $ to be the derived categories of $ \C^\times$-equivariant coherent sheaves on these varieties.

Cotangent bundles to Grassmannians have a particularly nice geometric description as
$$T^\star \bG(k,N) \cong \{(X,V): X \in \End(\C^N), 0 \subset V \subset \C^N, \dim(V) = k \text{ and } \C^N \xrightarrow{X} V \xrightarrow{X} 0 \}$$
where the notation $ \C^N \xrightarrow{X} V \xrightarrow{X} 0 $ means that $ X(\C^n) \subset V $ and that $ X (V)= 0$. Forgetting $X$ corresponds to the projection $T^\star \bG(k,N) \rightarrow \bG(k,N)$ while forgetting $V$ gives a resolution 
$$p_k: T^\star \bG(k,N) \xrightarrow{p_k} \{ X \in \End(\C^N) : X^2 =  0 \text{ and } \rank(X)  \le \min(k, N-k)\}.$$
Notice that $Y(\l) = T^\star \bG(k,N)$ we have the tautological vector bundle $V$ as well as the quotient bundle $\C^N/V$. In this picture, the action of $\C^\times$ on the fibres of $T^\star \bG(k,N)$ corresponds to acting by scalar multiplication on $X$. 

To describe the kernels $\sE$ and $\sF$ we use certain correspondences 
$$ W^r(\l) \subset T^\star \bG(k+r/2, N) \times T^\star \bG(k-r/2, N) $$ 
defined by
\begin{align*}
W^r(\l) := \{ (X,V,V') : &X \in \End(\C^N), \dim(V) = k + \frac{r}{2}, \dim(V') = k - \frac{r}{2}, 0 \subset V' \subset V \subset \C^N  \\
& \C^N \xrightarrow{X} V' \text{ and } V \xrightarrow{X} 0 \}
\end{align*}
(here, as before, $\l$ and $k $ are related by the equation $\l=N-2k$).

There are two natural projections $\pi_1: (X,V,V') \mapsto (X,V)$ and $\pi_2: (X,V,V') \mapsto (X,V')$ from $W^r(\l)$ to $Y(\l-r)$ and $Y(\l+r)$ respectively. Together they give us an embedding
$$(\pi_1, \pi_2): W^r(\l) \subset Y(\l-r) \times Y(\l+r).$$

On $W^r(\l)$ we have two natural tautological bundles, namely $V := \pi_1^*(V)$ and $V' := \pi_2^*(V)$ where the prime on the $V'$ indicates that the vector bundle is the pullback of the tautological bundle by the second projection. We also have natural inclusions
$$0 \subset V' \subset V \subset \C^N \cong \O_{W^r(\l)}^{\oplus N}.$$

We now define the kernel $\sE^{(r)}(\l) \in  D(Y(\l-r) \times Y(\l+r))$ by
\begin{equation*}
\sE^{(r)}(\l) := \O_{W^r(\l)} \otimes \det(\C^N/V)^{-r} \det(V')^r \{r(N-\l-r)/2\}.
\end{equation*}
Similarly, the kernel $\sF^{(r)}(\l) \in D(Y(\l+r) \times Y(\l-r))$ is defined by
\begin{equation*}
\sF^{(r)}(\l) := \O_{W^r(\l)} \otimes \det(V/V')^{\l} \{r(N+\l-r)/2\}.
\end{equation*}

These kernels induce functors by the formalism of Fourier-Mukai functors.  In \cite{ckl2}, we proved the following result. 

\begin{Theorem} \label{th:sl2action}
The functors induced by $\sE^{(r)}(\l)$ and $\sF^{(r)}(\l)$ define an enhanced strong categorical $\sl_2$ action. Here we have $ \la 1 \ra = [1] \{-1\} $ where $\{1\}$ is the shift of equivariant structure.
\end{Theorem}

Combining Theorem \ref{th:sl2action} with the main result of this paper we obtain:

\begin{Corollary} \label{cor:mainequiv} For any $ k \le N/2 $, we have an equivalence
\begin{equation*}
\T : \D(\l) = D(T^\star \bG(k, N)) \rightarrow D(T^\star \bG(N-k, N)) = \D(-\l)
\end{equation*}
induced by the convolution kernel of the complex $\Theta_*$ (here $\l = N - 2k$). 
\end{Corollary}

This equivalence restricts to an equivalence of the non-equivariant derived categories of coherent sheaves and for the remainder of the paper, we will ignore the equivariant structure (in order to simplify notation). Our goal will be to describe the induced equivalence in greater detail.

\subsection{The fibre product $Z(k)$} 

Fix $ k \le N/2 $. For any $ r \le N/2$, let $ B(r) $ be the nilpotent orbit 
\begin{equation*}
B(r) := \{ X \in \End(\C^N) : X^2 = 0, \text{ and } \rank(X) = r \}.
\end{equation*}
Then $ \overline{B(r)} = B(0) \cup \cdots \cup B(r) $.

Recall from above that $  T^\star \bG(k, N) \xrightarrow{p_k} \overline{B(k)} $ and $ T^\star \bG(N-k, N) \xrightarrow{p_{n-k}} \overline{B(k)}  $ are two different resolutions of $ \overline{B(k)} $.  Hence it is instructive to consider the fibre product
\begin{equation*}
Z(k) := T^\star \bG(k,N) \times_{B(k)} T^\star \bG(N-k,N) = \{ (X, V_1, V_2) : \C^N \xrightarrow{X} V_1, V_1 \xrightarrow{X} 0, \C^N \xrightarrow{X} V_2, V_2 \xrightarrow{X} 0 \}
\end{equation*}

This variety $ Z(k) $ is the union of the conormal bundles to the $ GL_N $ orbits on $ T^\star \bG(k, N) \times T^\star \bG(N-k, N) $.  In particular, it consists of $ k +1 $ equi-dimensional components.  Another way to describe the components is as follows. For $ s = 0, \dots, k$, define
\begin{equation*}
Z^s(k) := \overline{p_{k}^{-1}(B(s)) \times_{B(s)} p_{N-k}^{-1}(B(s))} \subset T^\star \bG(k,N) \times T^\star \bG(N-k,N)
\end{equation*}
These $ Z^s(k) $ are the components of $ Z(k) $.

\subsection{The terms in complex $\Theta_s$}
Consider the complex of kernels $\Theta_*$, with
$$\Theta_s:= \sF^{(\l+s)}(s) * \sE^{(s)}(\l+s)[-s]: D(T^\star \bG(k,N) \times T^\star \bG(N-k,N))$$
where $s=0, \dots, k$ (here, as before, $\l = N-2k$). 

The following result relates these kernels to the components of the fibre product described above.
\begin{Proposition}\label{prop:tpieces}
Fix $k,N$ as above. Then 
\begin{equation}\label{eq:tpieces}
\Theta_s \cong \O_{Z^{s}(k)} \otimes \det(\C^N/V)^{s} \det(V')^{s} [- s] \subset D(T^\star \bG(k,N) \times T^\star \bG(N-k,N))
\end{equation}
\end{Proposition}
In particular $ \Theta_s $ is a sheaf shifted into degree $ s $.

\begin{Remark} We do not actually know for sure that $Z^s(k)$ is normal so $\O_{Z^s(k)}$ in equation (\ref{eq:tpieces}) should really be the normalization of $\O_{Z^s(k)}$ (in other words $f_* \O_{\tilde{Z}^s(k)}$ where $f: \tilde{Z}^s(k) \rightarrow Z^s(k)$ is the normalization of $Z^s(k)$). In either case though $\Theta_s$ is a sheaf.
\end{Remark}

\begin{proof}
We will be working on the triple product $Y(\l) \times Y(\l+s) \times Y(-\l)$ where, as before, $\l = N-2k$ (so $Y(\l) = T^\star \bG(k,N)$). 

We need to compute
$$\Theta_s = \pi_{13*}(\pi_{12}^*(\O_{W^s(\l+s)} \otimes \det(\C^N/V)^{-s} \det(V')^s) \otimes \pi_{23}^*(\O_{W^{\l+s}(s)} \otimes \det(V'/V)^s)).$$
Fortunately, the intersection 
\begin{eqnarray*}
W 
&:=& \pi_{12}^{-1}(W^s(\l+s)) \cap \pi_{23}^{-1}(W^{\l+s}(s)) \in Y(\l) \times Y(\l+s) \times Y(-\l) \\
&=& \{0 \xrightarrow{k-s} V' \overset{s}{\underset{N-2k+s}{\rightrightarrows}} \begin{matrix} V \\ V'' \end{matrix} \overset{N-k}{\underset{k}{\rightrightarrows}} \C^N: XV \subset 0 \text{ and } XV'' \subset 0 \text{ and } X\C^N \subset V'\}
\end{eqnarray*}
is irreducible and of the expected dimension. To see this note that $\dim W^s(\l+s) = \frac{1}{2}(\dim Y(\l) + \dim Y(\l+2s))$ and similarly $\dim W^{\l+s}(s) = \frac{1}{2}(\dim Y(\l+2s) + \dim Y(-\l))$. So the expected dimension of the intersection $W$ is $\frac{1}{2}(\dim Y(\l) + \dim Y(-\l))$. On the other hand, $W$ has the resolution
$$W' := \{0 \xrightarrow{k-s} V' \overset{s}{\underset{N-2k+s}{\rightrightarrows}} \begin{matrix} V \\ V'' \end{matrix} \overset{N-2k+s}{\underset{s}{\rightrightarrows}} \tilde{V} \xrightarrow{k-s} \C^N: X \tilde{V} \subset 0 \text{ and } X\C^N \subset V'\}.$$
Now $W'$ fibres over $W^{N-2k+2s}(0)$ (which is smooth) by forgetting $V$ and $V''$. Thus $W'$ is smooth and irreducible of dimension $2s(l-k+s) + \dim W^{N-2k+2s}(0)$. This simplifies to give the same as the expected dimension calculated above. 

Consequently, 
\begin{eqnarray*}
\Theta_s &\cong& \pi_{13*}(\O_W \otimes \det(\C^N/V)^{-s} \det(V')^s \otimes \det(V''/V')^s) \\
&\cong& \pi_{13*}(\O_W \otimes \det(\C^N/V)^{-s} \det(V'')^s) \\
&\cong& \pi_{13*}(\O_W) \otimes \det(\C^N/V)^{-s} \det(V)^s
\end{eqnarray*}
where we obtain the last line by applying the projection formula ($\pi_{13}$ consists of forgetting $L_1'$ so $\det(\C^N/V)^{-s} \det(V)^s$ is a pullback from $Z^s(k)$). Now $W$ is a local complete intersection since it is the intersection of $\pi_{12}^{-1} W^s(\l+s)$ and $\pi_{23}^{-1} W^{\l+s}(s)$ which are both smooth. In particular, this means $W$ is normal. 

Now, in general, consider a map $p: A \rightarrow B$ where $B$ is a projective scheme. If for some very ample line bundle $M$ we have $H^i(p^* M) = 0$ for $i > 0$ then $p_* \O_A$ has no higher cohomology. This is because if $p_* \O_A$ had higher cohomology then $H^i(p^* M) \cong H^i(M \otimes p_* \O_A)$ would be non-zero for some $i>0$. The same argument works if $B$ is not necessarily projective but there is a projective map $f: B \rightarrow B'$ where $B'$ is affine. In the case $M$ can be taken to be relatively very ample with respect to $f$. 

Notice $W$ admits a projective map onto an affine variety by forgetting $V,V',V''$. So we can apply the above paragraph to the map $\pi': W' \rightarrow W$. By Lemma \ref{lem:can} the canonical bundle of $W'$ is the pullback of a line bundle from $W$. For simplicity denote this line bundle $M'$. Then 
$$H^i(\pi'^* M) \cong H^i(\omega_{W'} \otimes \pi'^*(M \otimes M'^\vee)).$$
Now $\pi'^*(M \otimes M'^\vee)$ is big and nef so by the Grauert-Riemenschneider vanishing theorem the right side vanishes. Thus $\pi'_*(\O_{W'})$ has no higher cohomology and we get
$$\pi'_*(\O_{W'}) \cong R^0 \pi'_*(\O_{W'}) \cong \O_W$$
where the second isomorphism follows since $W$ is normal. Thus to compute $\pi_{13*}(\O_W)$ we just need to compute $\pi_* \O_{W'}$ where $\pi: W' \rightarrow Z^s(k)$. Again by Lemma \ref{lem:can} we see that the canonical bundle of $W'$ is the pullback of a line bundle from $Z^s(k)$. Thus $\pi_{\O_{W'}}$ is also a sheaf and we get $\pi_* \O_{W'} \cong R^0 \pi_* \O_{W'}$. If $Z^s(k)$ were normal this would just be $\O_{Z^s(k)}$, otherwise it is the normalization of $\O_{Z^s(k)}$.  

\begin{Lemma}\label{lem:can} The canonical bundle of 
$$W' = \{0 \xrightarrow{k-s} V' \overset{s}{\underset{N-2k+s}{\rightrightarrows}} \begin{matrix} V \\ V'' \end{matrix} \overset{N-2k+s}{\underset{s}{\rightrightarrows}} \tilde{V} \xrightarrow{k-s} \C^N: X \tilde{V} \subset 0 \text{ and } X\C^N \subset V'\}.$$
is $(\det V \otimes \det V'')^{N-2k+2s}$. 
\end{Lemma}
\begin{proof}
Consider the projection map
$$W' \rightarrow \{0 \xrightarrow{k-s} V' \xrightarrow{N-2k+2s} \tilde{V} \xrightarrow{k-s} \C^N: X \tilde{V} \subset 0 \text{ and } X \C^N \subset V' \}$$
given by forgetting $V$ and $V''$. This is a Grassmannian bundle $\bG(s,\tilde{V}/V') \times \bG(N-2k+s, \tilde{V}/V')$ so that the relative canonical bundle is 
$$\left( \det(V/V')^{N-2k+s} \det(\tilde{V}/V)^{-s} \right) \otimes \left( \det(V''/V')^s \det(\tilde{V}/V'')^{-N+2k-s} \right).$$
Now 
$$\{0 \xrightarrow{k-s} V' \xrightarrow{N-2k+2s} \tilde{V} \xrightarrow{k-s} \C^N: X \tilde{V} \subset 0 \text{ and } X \C^N \subset V' \}$$
is carved out of 
$$\{0 \xrightarrow{k-s} V' \xrightarrow{N-2k+2s} \tilde{V} \xrightarrow{k-s} \C^N: X \tilde{V} \subset 0 \text{ and } X \C^N \subset \tilde{V} \}$$
via the section $X: \C^N/\tilde{V} \rightarrow \tilde{V}/V'$ and is in turn cut out of
$$\{0 \xrightarrow{k-s} V' \xrightarrow{N-2k+2s} \tilde{V} \xrightarrow{k-s} \C^N: X V' \subset 0 \text{ and } X \tilde{V} \subset V' \text{ and } X \C^N \subset \tilde{V}\}$$
via the section $X: \tilde{V}/V' \rightarrow V'$. In general, if $X \subset Y$ is cut out by a section of a vector bundle $U$ then 
$$\omega_X \cong \omega_Y \otimes \det U.$$
In our case the last space is the cotangent bundle to the partial flag variety $Fl(k-s,N-2k+2s,k-s;N)$ so the canonical bundle is trivial. So we get:
\begin{eqnarray*}
\omega_{W'} 
&\cong& \left( \det(V/V')^{N-2k+s} \det(\tilde{V}/V)^{-s} \right) \otimes \left( \det(V''/V')^s \det(\tilde{V}/V'')^{-N+2k-s} \right) \\
& & \otimes \left( \det(\C^N/\tilde{V})^{-N+2k-2s} \det(\tilde{V}/V')^{k-s} \right) \otimes \left( \det(\tilde{V}/V')^{-k+s} \det(V')^{N-2k+2s} \right). 
\end{eqnarray*}
Collecting terms and simplifying we are left with 
$$\omega_{W'} \cong \det(V)^{N-2k+s} \det(V'')^{N-2k+2s}.$$
\end{proof}

\end{proof}

\subsection{The equivalence}

By Theorem \ref{th:equiv}, the complex $ \Theta_* $ has a unique convolution and its cone  $\sT := \Cone(\Theta_*) \in D(T^\star \bG(k,N) \times T^\star \bG(N-k,N))$ induces the equivalence $ \T $.

\begin{Proposition}\label{prop:tsheaf}
The kernel $\sT \in D(T^\star \bG(k,N) \times T^\star \bG(N-k,N))$ is a sheaf with support $\operatorname{supp}(\sT) = Z(k)$. 
\end{Proposition}
\begin{proof}
Since $Z(k) = \bigcup_s Z^s(k)$ and $\operatorname{supp}(\Theta_s) = Z^s(k)$ it follows that $\operatorname{supp}(\sT) = Z(k)$. 

The less obvious part is that $\sT$ is a sheaf. To see this recall that $\sT$ is the convolution of the complex $\Theta_\bullet $.  Let $ B_\bullet $ be a Postnikov system with $ \sT = B_p$. 

We have a distiguished triangle 
\begin{equation*}
 B_0 = \Theta_0 \rightarrow B_1 \rightarrow \Theta_1[1]
\end{equation*}
and $\Theta_0$, $\Theta_1[1] $ are both sheaves.  Hence by the long exact sequence in cohomology $ B_1 $ is a sheaf.  

By the same argument we see that all $ B_i $ are sheaves and in particular $\sT = B_p $ is a sheaf.
\end{proof}

\subsection{Mukai flops}
We consider now the case $ k = 1 $. The varieties $T^\star \bG(1,N)$ and $T^\star \bG(N-1,N)$ are related by a standard Mukai flop via the diagram 
\begin{equation*}
\xymatrix{
& W^{N-2}(0) \ar[dl]_{\pi_1} \ar[dr]^{\pi_2} & \\ 
T^\star \bG(N-1,N) & & T^\star \bG(1,N) }
\end{equation*}
where both $\pi_1$ and $\pi_2$ are blowups in the zero sections. It was shown by Kawamata in \cite{kaw1} and Namikawa in \cite{nam1} that 
$$\pi_{2*} \circ \pi_1^*: D(T^\star \bG(1,N)) \rightarrow D(T^\star \bG(N-1,N))$$
does {\em not} induce an equivalence. But also in those papers they show that replacing $W^{N-2}(0)$ by the fibre product to get 
\begin{equation*}
\xymatrix{
& Z(1) = T^\star \bG(N-1,N) \times_{B(1)} T^\star \bG(1,N) \ar[dl]_{\pi'_1} \ar[dr]^{\pi'_2} & \\
T^\star \bG(N-1,N) \ar[dr] & & T^\star \bG(1,N) \ar[dl] \\
& B(1) & }
\end{equation*}
gives that $\pi'_{2*} \circ {\pi'}_1^*$ {\em is} an equivalence. In other words, the kernel $\O_{Z(1)}$ is an equivalence. 

Now, as noted above, our kernel $\sT = \Cone(\Theta_*)$ is also supported on $Z(1)$.  In this case, it is isomorphic to $\O_{Z(1)}$ (up to tensoring by pullbacks of line bundles from the two factors).  So we recover the same equivalence studied by Namikawa and Kawamata.

\subsection{The case of $ T^\star \bG(2,4)$}
Consider now $ N = 4 $ and $ k = 2 $.  This is the simplest situation not covered in the above section.

In contrast to the Mukai flop situation above, Namikawa showed in \cite{nam2} that the kernel 
$$\O_{Z(2)} \in D(T^\star \bG(2,4) \times T^\star \bG(2,4))$$
does {\em not} induce an equivalence.

 Subsequently, Kawamata in \cite{kaw2} was able to tweak the kernel $\O_{Z(2)}$ to obtain a new kernel (which is still a sheaf) and which does give an auto-equivalence of $D(T^\star \bG(2,4))$. However, we could not directly generalize his technique to more general $T^\star \bG(k,N)$. 

In \cite{C} we identify our kernel 
$$\sT \in D(T^\star \bG(k,N) \times T^\star \bG(N-k,N))$$ 
as the pushforward of an explicit line bundle from an open subset. Using this description we check that we get the same kernel as Kawamata's in the $T^\star \bG(2,4)$ case.

On the other hand, it is not hard to see that $\sT$ is {\em not} isomorphic to $\O_{Z(k)}$ for $k \ne 0, 1$ even after tensoring with pullbacks of line bundles from the two factors. This is because our kernels $\sT$ are always Cohen-Macaulay (CM). What this means is that their (derived) dual remains a sheaf (although it might be shifted in cohomological degree). The reason for this is that each component $\Theta_s$ is CM because it is (up to tensoring by a line bundle) the structure sheaf of some $Z^s(k)$ which is CM. Thus $(\Theta_*)^\vee$ is a complex made up of terms $(\Theta_s)^\vee$ which are sheaves (in the appropriate degrees) so that like in the proof of \ref{prop:tsheaf} its cone $\sT^\vee = \Cone((\Theta_*)^\vee)$ is also a sheaf. 

However, by Lemma \ref{lem:notCM} below, $Z(k)$ is not CM if $k \ne 0,1$. This perhaps explains (at least philosophically) why $\O_{Z(2)}$ did not induce an equivalence in the $T^\star \bG(2,4)$ example above. For this same reason, we suspect that $\O_{Z(k)}$ induces an equivalence {\em only} when $k=0,1$. 

\begin{Lemma}\label{lem:notCM} $Z(k)$ is not CM if $k \ne 0,1$. 
\end{Lemma}
\begin{proof}
The key point is that any two components $Z^s(k)$ and $Z^{s'}(k)$ (which have the same dimension) intersect inside $Z(k)$ in a locus which is codimension $\ge 2$ if $|s-s'| > 1$. 

So suppose $Z(k)$ contains more than two components (which is the case if $k \ne 0,1$) but is CM. We can take a partial normalization $\tilde{Z}(k)$ of $Z(k)$ so that all the components are disjoint except (say) $Z^0(k)$ and $Z^2(k)$. Since $Z(k)$ is CM then so is $\tilde{Z}(k)$. 

On the other hand, $\tilde{Z}(k)$ is smooth in codimension one (every $Z^s(k)$ is smooth in codimension one) but not normal (since we can still separate the components $Z^0(k)$ and $Z^2(k)$). Thus $\tilde{Z}(k)$ cannot satisfy Serre's S2 condition and thus cannot be CM (contradiction). 
\end{proof}


\begin{thebibliography}{E-G-S}

\bibitem[A]{A}
R. Anno, Spherical functors; \textsf{math.CT/0711.4409}.

\bibitem[C]{C}
S. Cautis, Equivalences and stratified flops, \textsf{arXiv:0909.0817}. 

\bibitem[CKL1]{ckl1} S. Cautis, J. Kamnitzer and A. Licata, Categorical geometric skew Howe duality, \textsf{arXiv:0902.1795}.

\bibitem[CKL2]{ckl2} S. Cautis, J. Kamnitzer and A. Licata, Coherent sheaves and categorical $\sl_2$ actions, \textsf{arXiv:0902.1796}.

\bibitem[CKL3]{ckl3}
S. Cautis, J. Kamnitzer and A. Licata, Quiver varieties and categorification, to appear. 

\bibitem[CK1]{ck1} S. Cautis and J. Kamnitzer, Knot homology via derived categories of coherent sheaves I, sl(2) case, \textit{Duke Math. J.} \textbf{142} (2008), no. 3, 511--588. \textsf{math.AG/0701194}

\bibitem[CR]{CR}
J. Chuang and R. Rouquier, Derived equivalences for symmetric groups and $\sl_2$-categorification, \textit{Ann. of Math.} \textbf{167} (2008), no. 1, 245--298; \textsf{math.RT/0407205}.

\bibitem[GM]{GM}
S. Gelfand and Y. Manin, \textit{Methods of homological algebra, second edition}, Springer-Verlag, 2003.

\bibitem[Ho]{Ho} 
R.P. Horja, Derived Category Automorphisms from Mirror Symmetry, \textit{Duke Math. J.} \textbf{127} (2005), 1--34; \textsf{math.AG/0103231}. 

\bibitem[Hu]{H}
D. Huybrechts, \textit{Fourier-Mukai Transforms in Algebraic Geometry}, Oxford University Press, 2006.

\bibitem[HT]{HT}
D. Huybrechts and R. Thomas, {$\p$-objects and autoequivalences of derived categories}; \textsf{math.AG/0507040}.

\bibitem[Ka1]{kaw1}
Y. Kawamata, D-equivalence and K-equivalence, \textit{J. Diff. Geom.}, \textbf{61} (2002), 147-171; \textsf{math.AG/0205287}.

\bibitem[Ka2]{kaw2}
Y. Kawamata, Derived equivalence for stratified Mukai flop on $\bG(2,4)$, \textit{Mirror symmetry. V}, AMS/IP Stud. Adv. Math., \textbf{38}, Amer. Math. Soc. (2006), 285--294; \textsf{math.AG/0503101}.

\bibitem[KT]{KT}
M. Khovanov and R. Thomas, Braid cobordisms, triangulated categories, and flag varieties, \textit{HHA} \textbf{9} (2007), 19--94; \textsf{math.QA/0609335}.

\bibitem[L]{l1}
A. Lauda, A categorification of quantum $\sl_2$, \textsf{arXiv:0803.3652v2}.

\bibitem[Lu]{Lu}
 G. Lusztig, \textit{Introduction to quantum groups.} Progress in Mathematics, 110. Birkh\"auser Boston, Inc., Boston, MA, 1993. xii+341 pp.

\bibitem[N1]{nam1}
Y. Namikawa, Mukai flops and derived categories. \textit{J. Reine. Angew. Math.} \textbf{560} (2003), 65--76; \textsf{math.AG/0203287}.

\bibitem[N2]{nam2}
Y. Namikawa, Mukai flops and derived categories II. \textit{Algebraic structures and moduli spaces}, CRM Proc. Lecture Notes, \textbf{38}, Amer. Math. Soc. (2004), 149--175; \textsf{math.AG/0305086}.

\bibitem[O]{O}
D. Orlov, Equivalences of derived categories and K3 surfaces, \textit{J. Math. Sci. (New York)} \textbf{84} (1997), no. 5, 1361--1381; \textsf{alg-geom/9606006}.

\bibitem[Rin]{Rin}
C. M. Ringel, Tame Algebras and Integral Quadratic Forms, \textit{Lecture Notes in Mathematics}, \textbf{1099} (1984), Springer Berlin.

\bibitem[R1]{RBourbaki}
R. Rouquier, Cat\'egories d\'eriv\'ee et g\'eometrie birationelles, \textit{Seminare Bourbaki}, 57eme ann\'ee, 2004--2005, no. 947.

\bibitem[R2]{R2}
R. Rouquier, Categorification of the braid groups; \textsf{math.RT/0409593}.

\bibitem[R3]{Rnew}
R. Rouquier, 2-Kac-Moody algebras; \textsf{math.RT/0812.5023}.

\bibitem[ST]{ST}
P. Seidel and R. Thomas, Braid group actions on derived categories of coherent sheaves, \textit{Duke Math J.}, \textbf{108}, (2001), 37--108; \textsf{math.AG/0001043}.

\bibitem[TL]{TL}
V. Toledano Laredo, A Kohno-Drinfeld theorem for quantum Weyl groups, \textit{Duke Math. J.} \textbf{112}, (2002), 421--451.
\end{thebibliography}
\end{document}